\documentclass {article}
\usepackage[english]{babel}
\usepackage[utf8]{inputenc}
\usepackage[T1]{fontenc}
\usepackage{graphicx,color}
\usepackage[final]{pdfpages}
\usepackage[extra]{tipa}
\usepackage{titlesec}
\usepackage{xspace}		
\usepackage[svgnames]{xcolor}
 \usepackage{lmodern}
\usepackage[hypertexnames=false]{hyperref}
\usepackage[margin=1in]{geometry}

\usepackage{mathtools}
\usepackage{graphicx}
\usepackage[ruled,vlined]{algorithm2e}
\usepackage{amsmath,amsfonts}
\usepackage{dsfont}
\usepackage{amssymb}
\usepackage[ntheorem]{empheq}
\usepackage{mathenv}
\usepackage[standard,thref,amsmath,thmmarks]{ntheorem}
\usepackage{nicefrac}

\usepackage{theoremref}

\usepackage{caption,subcaption}
\usepackage{color}
\usepackage{comment}
\usepackage{amsfonts}
\usepackage{bm}

\usepackage{minitoc}
\usepackage{cite}
\mathtoolsset{showonlyrefs,showmanualtags}

\numberwithin{counter}{subsection}

\usepackage{aliascnt}

\renewtheorem{theorem}{Theorem}[section]

\newaliascnt{lem}{theorem}
\renewtheorem{lemma}[lem]{Lemma}
\aliascntresetthe{lem}

\newaliascnt{def}{theorem}
\renewtheorem{definition}[def]{Definition}
\aliascntresetthe{def}

\newaliascnt{prop}{theorem}
\renewtheorem{proposition}[prop]{Proposition}
\aliascntresetthe{prop}

\newaliascnt{coro}{theorem}
\renewtheorem{corollary}[coro]{Corollary}
\aliascntresetthe{coro}

\newaliascnt{rmk}{theorem}
\renewtheorem{remark}[rmk]{Remark}
\aliascntresetthe{rmk}

\newtheorem*{property*}{Property:}
\newtheorem*{conjecture}{Conjecture:}
\newtheorem{question}[Theorem]{Question}

\newcommand{\D}{\partial}
\newcommand{\eps}{\varepsilon}
\newcommand{\R}{\mathbb{R}}
\newcommand{\la}{\lambda}

\begin{document}
\title{Symmetry properties of stable solutions of semilinear elliptic equations in unbounded domains}


\author{Samuel Nordmann\thanks{samuel.nordmann@ehess.fr ; Ecole des Hautes Etudes en Sciences Sociales, PSL University, CNRS, Centre d'Analyse et Mathematiques Sociales. 54 boulevard Raspail, 75006 Paris. ORCID ID: 0000-0002-3562-6893.
}}



\maketitle
\begin{abstract}
We consider stable solutions of a semilinear elliptic equation with homogeneous Neumann boundary conditions.
A classical result of Casten, Holland and Matano states that all stable solutions are constant in convex bounded domains. In this paper, we examine whether this result extends to \emph{unbounded} convex domains. We give a positive answer for {stable non-degenerate} solutions, and for stable solutions if the domain $\Omega$ further satisfies $\Omega\cap\{\vert x\vert\leq R\}= O(R^2)$, when $R\to+\infty$. If the domain is a straight cylinder, an additional natural assumption is needed. These results can be seen as an extension to more general domains of some results on De Giorgi's conjecture.

As an application, we establish asymptotic symmetries for stable solutions when the domain satisfies a geometric property asymptotically. 

\end{abstract}
\paragraph{Keywords:} Semilinear elliptic equations ; Stability ; Symmetry ; Neumann boundary conditions ; De Giorgi's conjecture ; Liouville property ; Convex domains ; Unbounded domains ; Generalized principal eigenvalue\\

\noindent {\bf AMS Class. No:} 35B35, 35B06, 35J15, 35J61, 35B53.
\paragraph{Acknowledgement.}
The author is deeply thankful to Professor Henri Berestycki for proposing the subject and for all the very instructive discussions.\\
The author also thanks the anonymous referee for his comments which helped to improve the paper.\\
The research leading to these results has received funding from the European
Research Council under the European Union's Seventh Framework Programme
(FP/2007-2013) / ERC Grant Agreement n.321186 - ReaDi -ReactionDiffusion
Equations, Propagation and Modelling held by Henri Berestycki.

\tableofcontents

\section{Introduction}

\subsection{Presentation of the problem}
We study some symmetry properties of stable solutions of semilinear elliptic equations with Neumann boundary conditions. We consider the following problem:
\begin{equation}\label{ANG_Intro_EquationSemilineaire}
    \left\{
    \begin{aligned}
        &-\Delta u(x)=f(u(x)) &&\forall x\in\Omega,\\ 
        &\D_\nu u(x)=0 &&\forall x\in\D\Omega,\\
        &u\in C^{3}\left(\overline{\Omega}\right)\quad ;\quad \nabla u \in L^\infty(\Omega),
    \end{aligned}
    \right.
\end{equation}
where $\Omega\subset\R^n$ is a (possibly unbounded) smooth domain, $\D_\nu$ is the outward normal derivative, and $f$ is $C^{1}$.

A solution is said to be \emph{stable} if the second variation of energy at $u$ is nonnegative. 
\begin{definition}\label{ANG_Intro_DefStabilite}
Let $u$ be a solution of~\eqref{ANG_Intro_EquationSemilineaire} and set
\begin{equation}\label{DefLambda}
\la_1:=\inf\limits_{\substack{\psi\in C^1_0(\overline{\Omega})\\ \Vert \psi\Vert_{{L}^2}=1}}\int_\Omega\vert\nabla \psi\vert^2-f'(u)\psi^2=: \inf\limits_{\substack{\psi\in C^1_0(\overline{\Omega})\\ \Vert \psi\Vert_{{L}^2}=1}} \mathcal{F}(\psi),
\end{equation} 
where $C^1_0(\overline{\Omega})$ is the space of continuously differentiable functions with compact support in $\overline{\Omega}$ (which do not necessarily vanish on $\D\Omega$).

The solution $u$ is said to be stable if $\lambda_1\geq0$, and stable non-degenerate if $\lambda_1>0$.
\end{definition}
A stable non-degenerate solution is then a (non-degenerate) minimum of the energy. On the other hand, any degenerate critical point of the energy is \emph{stable} according to our definition. See Appendix~\ref{Annexe:Stability} for more detailed on the link between this definition and the classical dynamical definition of stability.

Note that if $z\in\R$ is a stable root of $f$, i.e., $f(z)=0$ and $f'(z)\leq0$, it is a (trivial) stable solution. 
We are interested in the existence/non-existence and symmetries of \emph{non-trivial} stable solutions, called \emph{patterns} in the sequel.
\begin{definition}
We call  \emph{pattern} (resp. non-degenerate pattern) any non-constant stable (resp. stable non-degenerate) solution.
\end{definition}
In two independent papers, Casten, Holland~\cite{Casten1978a}, and Matano~\cite{Matano1979} proved the following result. 
\begin{theorem}[\cite{Casten1978a,Matano1979}]\label{ANG_th:Intro_CHM}
If the domain $\Omega$ is bounded and convex, there exists no pattern to~\eqref{ANG_Intro_EquationSemilineaire}.
\end{theorem}
We insist on the fact that this conclusion is valid for any $f\in C^{1}$.

The main purpose of this paper is to examine whether the above theorem extends to \emph{unbounded} convex domains.
The classification of stable solutions in the particular case $\Omega=\R^n$ has already been widely investigated: this problem is intricate, and closely linked to De Giorgi's conjecture~\cite{Wei2018}.
However, it appears that the literature only deals, on the one hand, with bounded convex domains and, on the other hand, with the entire space $\R^n $. It does not deal, or only slightly, with general unbound convex domains.

\paragraph*{}
In this paper, we establish the non-existence of non-degenerate patterns in any convex unbounded domain, and the $1$-dimensional symmetry of (possibly degenerate) patterns if the domain $\Omega$ further satisfies $\limsup_{R\to+\infty}\frac{\vert \Omega\cap\{\vert x\vert\leq R\}\vert}{R^2}<+\infty$. In such domains, this latter result implies the non-existence of patterns if $\Omega$ is not a straight cylinder, or if $f$ satisfies an additional natural assumption. One can see those results as extending some advances on De Giorgi's conjecture to a more general class of domains

As an application, we establish asymptotic symmetries for patterns when the domain satisfies a geometric property asymptotically. In particular, if the domain is a cylinder (with a varying section) that tends to be convex at infinity, we prove that a pattern must converge to a constant. 

For the reader's convenience, we give a simple proof of \autoref{ANG_th:Intro_CHM} in section~\ref{Sec:FormalApproach}.
We also recall further classical symmetry results in section~\ref{Sec:SymmetryProperties}, namely that stable solutions inherit from the domain's invariance with respect to translation or rotations. The appendix proposes a discussion on the notion of generalized principal eigenvalue, the different definitions of stability, and the isolation of stable solutions. In a forthcoming paper~\cite{Nordmann2019b}, we will investigate to what extent one can relax the assumption that the domain is convex in \autoref{ANG_th:Intro_CHM}.

\subsection{Context, general remarks, and references}

Consider the energy 
\begin{equation}\label{Energy_eps}
\mathcal{E}_\eps (u)= \int_\Omega{\eps^2}\vert\nabla u\vert^2 + F(u),
\end{equation}
where $\Omega$ is bounded, $\eps>0$ is a parameter and $F$ is a two-well potential, say, $F(u):=\frac{1}{2}\left(1-u^2\right)^2$.
Any minimizer of $\mathcal{E}_\eps$ is a stable solution of the associated Euler-Lagrange equation:
\begin{equation}\label{ANG_Intro_EquationEps}
\left\{\begin{aligned}
&-\eps^2 \Delta u_\eps = f(u_\eps)&&\text{in }\Omega,\\
&\D_\nu u_\eps=0 &&\text{on }\D\Omega,
\end{aligned} \right.
\end{equation}
with Allen-Cahn's nonlinearity $f(u):=u-u^3$.

A series of seminal papers~\cite{Modica1979,Caffarelli1995,Caffarelli2006,Kohn1989} establishes that when $\eps\to0$, the level sets of patterns of \eqref{ANG_Intro_EquationEps} converge to minimal surfaces in $\Omega$.
From a rescaling $x\leftrightarrow \eps x$, i.e., zooming around the origin, equation~\eqref{ANG_Intro_EquationEps} at the limit $\eps\to0$ reduces to equation~\eqref{ANG_Intro_EquationSemilineaire} in $\R^n$. According to the above, the level sets of stable solutions in $\R^n$ should be minimal surfaces in $\R^n$.
These are known to be necessarily hyperplanes if and only if $n\leq7$~\cite{Simons1968,Bombieri1969}. Regarding a minimal surface which is also the graph of a function defined on $\R^{n-1}$, we gain one dimension: such a surface is necessarily a hyperplane if and only if $n\leq8$~\cite{DeGiorgi1965,Bombieri1969,Cabre,Jerison2004}.
This brought De Giorgi to state the following conjecture.
\begin{conjecture}[De Giorgi]
Let $u$ be a solution of $-\Delta u=u-u^3$ in $\R^n$, such that $\vert u\vert<1$ and $\D_{x_n}u>0$. The level sets of $u$  are hyperplanes, at least if $n\leq 8$.
\end{conjecture}
The fact that the level sets of $u$ are hyperplanes means that $u$ is \emph{flat}, {i.e.}, $u$ is $1$-d symmetric. It should be noted that the assumption $\D_{x_n} u>0$ implies that every level set of $u$ is the graph of a function defined on $\R^{n-1}$, which explains why the conjecture is stated for $n\leq 8$ and not $n\leq 7$.

A vast litterature is devoted to De Giorgi's conjecture (see~\cite{Wei2018,Farinaa} for a state of the art). The conjecture is proved by Ghoussoub and Gui~\cite{Ghoussoub1998} in dimension $n\leq 2$, by Ambrosio and Cabré~\cite{Ambrosio2000} in dimension $n=3$, and by Savin~\cite{Savin2003,Savin2009} in dimension $4\leq n\leq 8$ under the additional assumption
\begin{equation}\label{ANG_Intro_HypotheseSavin}
\lim\limits_{x_n\to\pm\infty} u(x'',x_n)=\pm1.
\end{equation}
A counter example is provided by del Pino, Kowalczyk and Wei~\cite{DelPino2008} in dimension $n\geq 9$ (this counterexample satisfies~\eqref{ANG_Intro_HypotheseSavin}). The conjecture is still open for dimensions $4\leq n\leq 8$ without the additional assumption~\eqref{ANG_Intro_HypotheseSavin}. Note that the proofs are, in general, valid for a rather large class of nonlinearities, and not only for Allen-Cahn's nonlinearity~\cite{Cabre,Alberti2001,Savin2009}.
Analogous results have also been obtained for non-compact Riemannian manifolds without boundary~\cite{Farina2013a,Farina2013b}..

\paragraph*{}
The assumption $\D_{x_n}u>0$ in De Giorgi's conjecture implies in particular that $u$ is stable (it is a pattern). Indeed, setting $v:=\D_{x_n}u$, and differentiating~\eqref{ANG_Intro_EquationSemilineaire}, we get $-\Delta v-f'(u)v=0$ in $\R^n$. Multiplying this equation by $\frac{\psi^2}{v}$ (for a test function $\psi$ with compact support), integrating on $\R^n$, and using the divergence theorem, we obtain
\begin{align*}
0&=\int_{\R^n}\nabla v\cdot\nabla\frac{ \psi^2}{ v}-f'(u) \psi^2\\
&=\int_{\R^n}2\frac{ \psi}{ v}\nabla v\cdot\nabla \psi-\frac{ \psi^2}{ v^2}\left\vert\nabla v\right\vert^2-f'(u) \psi^2\\
&\leq \int_{\R^n} \left\vert\nabla \psi\right\vert^2-f'(u) \psi^2,
\end{align*}
where we use Young's inequality $2ab \leq a^2 + b^2$ in the last step. We deduce $\lambda_1\geq0$, {i.e.}, $u$ is a pattern.

It is therefore natural to consider a variant of De Giorgi's conjecture, replacing the assumption $\D_{x_n}u>0$ with the assumption that $u$ is a pattern.
\begin{conjecture}[De Giorgi's variant]
Let $u$ be a pattern of $-\Delta u=u-u^3$ in $\R^n$, such that $\vert u\vert<1$. The level sets of $u$ are hyperplanes, at least if $n\leq 7$.
\end{conjecture}
Note that with the stability assumption (instead of $\D_{x_n} u>0$), we cannot guarantee that every level set of $u$ is the graph of a function defined on $\R^{n-1}$. This explains why this conjecture is stated for $n\leq7$ (unlike De Giorgi's conjecture, stated for $n\leq8$).

This variant is proved by Ghoussoub and Gui~\cite{Ghoussoub1998} for $n\leq2$, and a counterexample is known to exists in dimension $n=8$
\footnote{If the variant of De Giorgi conjecture were to be true in dimension $n=8$, the method of Ambrosio and Cabré~\cite{Ambrosio2000} would imply that De Giorgi conjecture holds in dimension $n=9$ which is impossible from the result of del Pino, Kowalczyk and Wei~\cite{DelPino2008}. See also~\cite{Pacard2013}. }, proving that the condition $n\leq7$ is optimal. The intermediate dimensions $3\leq n\leq 7$ are open.
Nevertheless, Dancer~\cite{Dancer2004} proves the non-existence of \emph{non-degenerate} patterns in $\R^n$ in any dimension.

\paragraph*{}
Let us also mention the results of~\cite{Cabrea} (refined in~\cite{Villegas2007}) which establish that radial patterns exist in $\R^n$ if and only if $n\geq11$.
Many further results (including regularity of weak stable solutions) are available under stronger assumptions on $f$, for example, if $f$ is increasing, convex, positive, see~\cite{Dupaigne2011,Cabreb} and references therein.

\subsection{Main results}

Our first result establishes the non-existence of non-degenerate patterns in any unbounded convex domain.
\begin{theorem}\label{thmStableNonDegenerate}
There exists no non-degenerate pattern to~\eqref{ANG_Intro_EquationSemilineaire} in convex (possibly unbounded) domains.
\end{theorem}
This theorem has been proved by Dancer~\cite{Dancer2004} when $\Omega=\R^n$. In this case, the result can be formally justified as follows: the stability of a given pattern in $\R^n$ can be nothing but degenerate, since a continuum of patterns is given by translations of this same pattern. This argument is no longer valid if $\Omega$ is not invariant by translation.

Note that it is assumed in ~\eqref{ANG_Intro_EquationSemilineaire} that the solution has a bounded gradient. If we drop this assumption, then \autoref{thmStableNonDegenerate} do no longer apply, since $u(x):=e^x$ is a non-degenerate pattern of $-u''=-u$ in $\R$.
%

\paragraph*{}
The following result gives a classification of \emph{possibly degenerate} patterns when the domain is convex and satisfies a growth condition at infinity.

\begin{theorem}\label{thmdimension2}
Assume that $\Omega$ is convex and satisfies
\begin{equation}\label{DomaineGeneral}
\limsup_{R\to+\infty}\frac{\vert \Omega\cap\{\vert x\vert\leq R\}\vert}{R^2}<+\infty.
\end{equation}
Let $u$ be a stable solution of~\eqref{ANG_Intro_EquationSemilineaire}.
\begin{enumerate}
\item If $\Omega$ is not a straight cylinder ({i.e.}, $\Omega$ is not of the form $\R\times\omega$, $\omega \subset\R^{n-1}$), then $u$ is constant.
\item If $\Omega$ is a straight cylinder, then $u$ is either constant or a monotonic flat solution. If, in addition, $u$ is bounded then it connects two stable roots $(z^-,z^+)$ of $f$ such that ${\int_{z^-}^{z^+} f =0}$.
\end{enumerate}
\end{theorem}
If $\Omega$ satisfies~\eqref{DomaineGeneral}, the above result ensures the non-existence of patterns if either $\Omega$ is not a straight cylinder, or $\int_{z_1}^{z_2} f\neq 0$ for all $z_1\neq z_2$ in $\left\{z\in\R:f(z)=0\text{ and } f'(z)\leq0\right\}$.

Note that, in the special case of a straight cylinder $\Omega=\R\times\omega$, the convexity of the domain is degenerate in one direction. Planar patterns may indeed exist in such domains. For example, the Allen-Cahn equation in $\R$, $-u''=u(1-u^2)$
admits the explicit solution 
$u : x\mapsto \tanh\frac{x}{\sqrt{2}}$ which is stable (degenerate). This solution can be seen as a traveling wave connecting the two stable roots $\pm1$, with speed $c=\int_{-1}^1 f=0$.

Condition~\eqref{DomaineGeneral} is usual in our context. It echoes with the celebrated Liouville type property of Berestycki, Caffarelli and Nirenberg (Theorem~1.7 in~\cite{Berestycki1997b}), of which a refined version is presented in \autoref{schrodinger2}. 
Under this assumption, we allow for instance convex domains that are subdomains of $\R^2$, or of the form $\R^i\times\omega$ with $\omega$ bounded and $i\in\{1,2\}$, or of the form
\begin{equation}
\Omega = \left\{(x_1,x')\in\R\times\R^{n-1}: x'\in\omega(x_1)\right\},
\end{equation}
where for all $x_1\in\R$, $\omega(x_1)\subset\R^{n-1}$ with $\limsup\limits_{\vert x_1\vert\to+\infty}\frac{\vert \omega(x_1)\vert}{\vert x_1\vert}<+\infty$.

According to De Giorgi's conjecture, it is reasonable to think that condition~\eqref{DomaineGeneral} can be substantially weakened, but not completely dropped since non-planar patterns may exist in $\R^8$. This problem has been widely investigated but many questions remain open, see section~\ref{Sec:schrodinger} for more details. Note, however, that straight cylinders (or the entire space) are domains for which convexity is degenerate. Up to the author's knowledge, not much is known about the classification of patterns in strictly convex domains.
\begin{question}
Can we relax assumption~\eqref{DomaineGeneral} under the assumption that the domain is strictly convex?
\end{question}

The conclusions, assumptions, and the proof of~\autoref{thmdimension2} are closely related to the work of Farina, Mari and Valdinoci~\cite{Farina2013b}. In Theorem~1 there, the authors suppose the existence of a non-constant stable solutions of $-\Delta u=f(u)$ with bounded gradient on a Riemannian manifold $M$ without boundary and with nonnegative Ricci curvature. Without further assumptions on $u$, it is also assumed either that the manifold is parabolic (i.e. it does not admit any Green function) or that the domain satisfies a growth condition similar to~\eqref{DomaineGeneral} but replacing $R^2$ by $R^2\log R$ (we prove that \autoref{thmdimension2} actually holds under this weaker condition if the domain can be written as a cylinder with varying cross section, see the end of Section~\ref{Sec:schrodinger}). Then, the authors in~\cite{Farina2013b} prove that the manifold can be written $M=\R\times N$ and that $u$ only depends on the first variable. These two conclusions correspond respectively to the two conclusions of \autoref{thmdimension2}.
Note that, since $M$ is assumed to have no boundary, the only instance of $M$ which is also a euclidian domain is the whole space $\R^n$.

\paragraph*{}

In section~\ref{Sec:SymmetryProperties}, we give further symmetry properties of patterns in unbounded domains, namely that patterns inherit the domain's invariance with respect to translations and planar rotations.

\paragraph{}
As an application of our results, we propose to establish an \emph{asymptotic} variant of~\autoref{ANG_th:Intro_CHM}. We consider a cylindrical domain which is \emph{asymptotically} convex (see Figure~\ref{ANG_IntG_fig:AsymptoticCylinder}), and show that any pattern \emph{converges} to a constant.
\begin{figure}[!h]
  \centering
    \centering
    \includegraphics[width=0.4\textwidth]{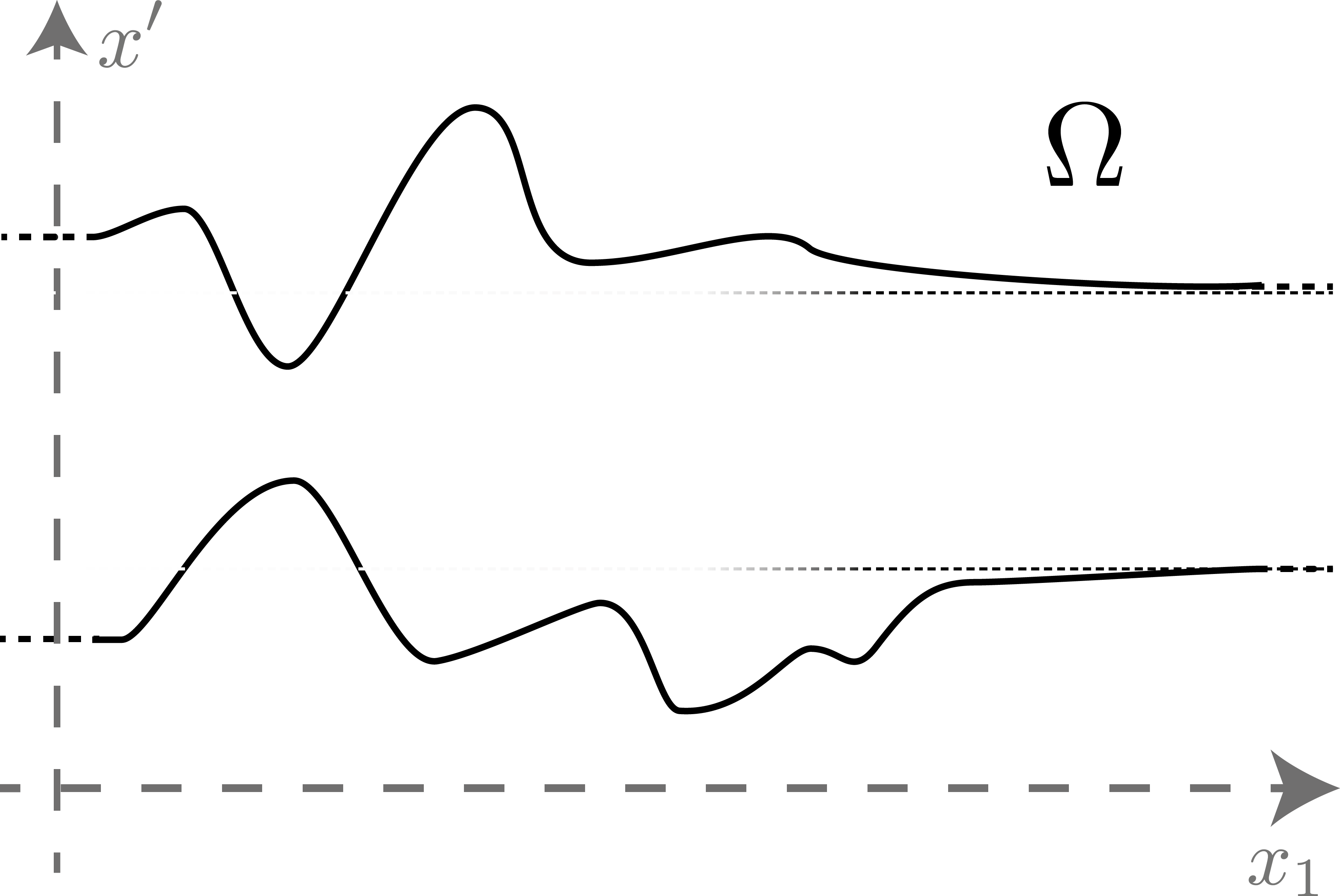}
    \hfill
    \includegraphics[width=0.4\textwidth]{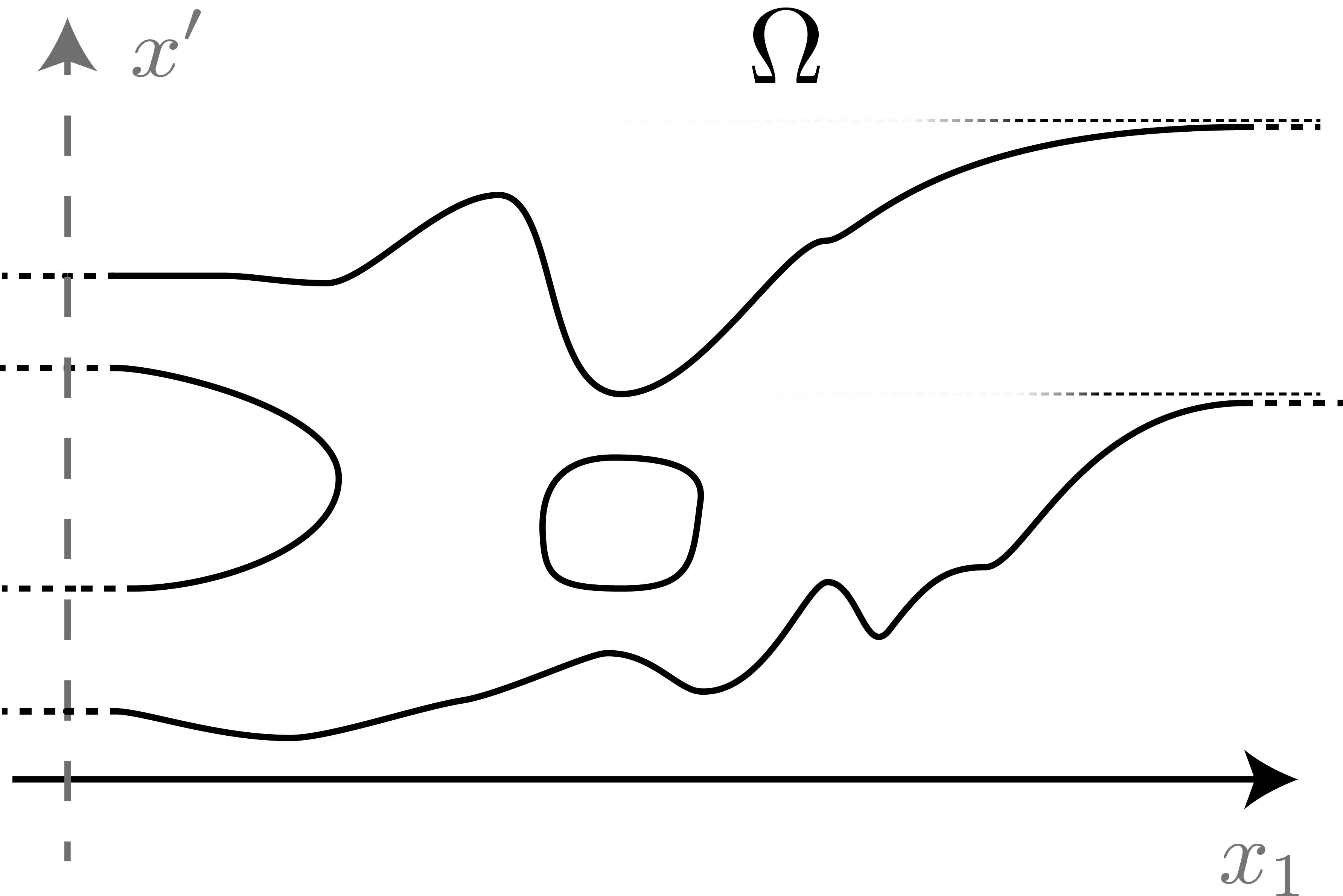}
    \caption{Examples of asymptotically convex cylinder\label{ANG_IntG_fig:AsymptoticCylinder}}
\end{figure}
The following result deals with non-degenerate patterns.
\begin{theorem}\label{th:AsymptoticNonDegenerate}
Let $\Omega\subset\R^N$ be a uniformly smooth cylindrical domain (with varying section) on the $x_1$ axis, which converges to a straight cylinder $\Omega_\infty:=\R\times\omega_\infty$ when $x_1\to+\infty$ (\autoref{DefConvergenceDomain}). Suppose that $\omega_\infty$ is convex, and let $u$ be a bounded non-degenerate pattern of~\eqref{ANG_Intro_EquationSemilineaire}. Then, $u(x_1,\cdot)$ converges $C^2_{loc}$ to a stable root of $f$ when $x_1\to+\infty$. 
\end{theorem}
The case of possibly degenerate patterns is treated in \autoref{th:AsymptoticDegenerate}.
Section~\ref{sec:Symmetry} contains other symmetry results, if the limiting domain $\Omega_\infty$ is invariant under a translation or a rotation. In particular, we prove the following.
\begin{corollary}\label{th:AsymptoticCoro}
Let $\Omega\subset\R^{N}$ be a cylindrical domain (with a variable section) on the $x_1$ axis, which converges to a straight cylinder $\Omega_\infty:=\R\times\omega_\infty$, with $w_\infty\subset\R^{n-1}$ bounded, when $x_1\to+\infty$ (\autoref{DefConvergenceDomain}). Let $u$ be a bounded stable non-degenerate solution of~\eqref{ANG_Intro_EquationSemilineaire}. Then, when $x_1\to+\infty$, $u(x_1,\cdot)$ converges $C^2_{loc}$ to $u_\infty(\cdot)$, which is a stable solution in the section $\omega_\infty$:
\begin{equation}\label{EquationInSection}
\left\{\begin{aligned}
&-\Delta_{x'}u_\infty=f(u_\infty) &&\text{in }\omega_\infty,\\
&\D_\nu u_\infty=0 &&\text{on }\D\omega_\infty.
\end{aligned}\right.\end{equation}
\end{corollary}
The asymptotic symmetry results presented in Section~\ref{sec:Symmetry} actually hold when the stability of the solution is only assumed to hold outside a compact set, see~\autoref{Rmk:StabilityOutsideCompact}.

\subsection{Outline of the paper}
Section~\ref{Sec:FormalApproach} is devoted to some preliminaries and the classical proof of~\autoref{ANG_th:Intro_CHM}.
We study the non-existence of patterns in unbounded domains in section~\ref{sec:UnboundedDomains}, that is, we prove \autoref{thmStableNonDegenerate}, \autoref{thmdimension2}, and also recall classical further symmetry properties for patterns when the domain is invariant under translations or rotations.

In section~\ref{sec:AsymptoticFormulation} we establish asymptotic symmetries for patterns when the domain satisfies a geometrical property asymptotically. We prove \autoref{th:AsymptoticNonDegenerate} and give further properties.

In the Appendix~\ref{sec:GeneralizedPrincipalEigenvalue}, we define and study the notion of generalized principal eigenvalue in unbounded domains. We also discuss the different definitions of stability in Appendix~\ref{Annexe:Stability}, and the isolation of stable solutions in Appendix~\ref{sec:Isolation}.

\section{Preliminaries: the classical case of bounded convex domains}\label{Sec:FormalApproach}
To give an overview of the method and the difficulties arising when dealing with unbounded domains, we recall the classical proof of \autoref{ANG_th:Intro_CHM} which deals with bounded domains. It also allows us to introduce a key lemma.
Let $\Omega$ be a smooth, convex, bounded domain, $u$ a stable solution of \eqref{ANG_Intro_EquationSemilineaire} and set $v_i:=\D_{x_i} u$, for all $i\in\{1,\dots,n\}$. \\

\noindent\textbf{Step 1.} On the one hand, differentiating \eqref{ANG_Intro_EquationSemilineaire} with respect to $x_i$, we find that $v_i:=\D_{x_i}u$ satisfies the linearized equation
\begin{equation}\label{linearization}
    -\Delta v_i - f'(u)v_i=0\quad\text{in }\Omega.
\end{equation}
From an integration by part, we obtain
\begin{equation*}
\mathcal{F}(v_i)=\int_{\D\Omega}v_i\D_\nu v_i=\frac{1}{2}\int_{\D\Omega}\D_\nu v_i^2,
\end{equation*}
with $\mathcal{F}$ from \eqref{DefLambda}. On the other hand, as $u$ is stable, we have $\mathcal{F}(\cdot)\geq0$ and
\begin{equation}\label{SommeDeF}
0\leq \mathcal{F}(v_i)\leq\sum\limits_{k=1}^n \mathcal{F}(v_k)=\frac{1}{2}\int_{\D\Omega}\D_\nu\vert\nabla u\vert^2.
\end{equation}
(If $\Omega$ is unbounded, the computations are not licit and need to be adapted, see section~\ref{Sec:ProofNonDegenerate}.)
\\

\noindent\textbf{Step 2.}
When the domain is convex, the above integral turns out to be nonpositive, as stated in the following key lemma. This is where the \emph{convexity} of the domain comes into play.  It can be found in \cite{Casten1978a,Matano1979a}, but a simple proof is presented at the end of the section for completeness. Note that the lemma remains valid if the domain is unbounded.
\begin{lemma}[\cite{Casten1978a,Matano1979a}] \label{LemmeIntermediaireVariations}
Let $\Omega$ be a smooth convex domain. If $u$ is a $\mathcal{C}^2$ function such that
\begin{equation}\label{Neumann}
\D_\nu u=0\quad \text{on }\D\Omega,
\end{equation}
then 
\begin{equation*}
\D_\nu\vert\nabla u\vert^2\leq 0\quad \text{on }\D\Omega.
\end{equation*}
\end{lemma}
From the above lemma and the inequality~\eqref{SommeDeF}, we conclude that for all $i\in\{1,\dots,n\}$, we have $\mathcal{F}(v_i)=0$. At this step, if $u$ is not constant (i.e. $v_i\not\equiv0$ for some $i$), then we get $\lambda_1=0$, i.e., $u$ is \emph{stable degenerate}.
\\

\noindent\textbf{Step 3.} Since $\mathcal{F}(v_i)=0$ and $\lambda_1\geq0$, we deduce that $v_i$ minimizes $\mathcal{F}$. It is then classical that $v_i$ is a multiple of the eigenfunction associated to $\lambda_1$, denoted $\varphi$, which is unique (up to a multiplicative constant) and positive in $\overline{\Omega}$. 

(This conclusion may fail in unbounded domains, see section~\ref{Sec:schrodinger}.)
\\

\noindent\textbf{Step 4.} From $\D_\nu u=0$ on the closed surface $\D\Omega$, we deduce that $v_i$ vanishes on some point of the boundary. But as $v_i$ is colinear to $\varphi$, we conclude $v_i\equiv0$, which completes the proof.

(If $\Omega$ is a straight cylinder, the above conclusion fails and $v_i$ may be a nonzero multiple of $\varphi$.)

\paragraph*{}Before proving \autoref{LemmeIntermediaireVariations}, we need the following definition.
\begin{definition}\label{DefRepresentationofTheBoundary}
Let $\Omega\subset\R^n$. A ``representation of the boundary'' is a pair $(\rho, U)$ where $\rho$ is a $C^{2}$ function defined on $U$ a neighborhood of  $\D\Omega$ such that 
\begin{align*}
&\rho(x)
\begin{cases}
<0 &if $x\in\Omega\cap U$\\
=0 &if $x\in\D\Omega$\\
>0 &if $x\in U\backslash\overline\Omega$
\end{cases}  &\text{and }\quad\nabla\rho(x)=\nu(x)\quad \forall x\in\D\Omega,
\end{align*}
where $\nu(x)$ is the outer normal unit vector of $\D\Omega$ at $x$.
\end{definition}
It is classical that such a representation of the boundary always exists for $C^{2,1}$ domains, see e.g. section 6.2 of~\cite{GilbarDavid2015}.

\begin{proof}[of \autoref{LemmeIntermediaireVariations}]
Let us consider $(\rho,U)$ a representation of the boundary for $\Omega$. Equation~\eqref{Neumann} becomes
\begin{equation}
\nabla u\cdot\nabla\rho=0\quad\text{on }\D\Omega.
\end{equation}
As $\nabla u$ is tangential to $\D\Omega$, we can differentiate the above equality with respect to the vector field $\nabla u$. It gives, on $\D\Omega$,
\begin{equation*}
0=\nabla\left(\nabla u\cdot\nabla\rho\right)\cdot\nabla u=\nabla u\cdot\nabla^2 u\cdot\nabla\rho+\nabla u\cdot\nabla^2\rho\cdot\nabla u.
\end{equation*}
From this, we infer
\begin{equation*}
\D_\nu\vert\nabla u\vert^2=\nabla\left(\vert\nabla u\vert^2\right)\cdot\nabla\rho=2\nabla u\cdot\nabla^2 u\cdot\nabla\rho=-2\nabla u\cdot\nabla^2\rho\cdot\nabla u,
\end{equation*}
Since $\Omega$ is convex, for all $x_0\in\D\Omega$, we have that $\nabla^2\rho(x_0)$ is a nonnegative quadratic form in the tangent space of $\D\Omega$ at $x_0$.
As $\nabla u$ is tangential to $\D\Omega$, we deduce from the above equation that $\D_\nu\vert\nabla u\vert^2$ is nonpositive.
\end{proof}
In \cite{Casten1978a}, the authors give the following remarkable geometrical interpretation of the above lemma. Consider a bounded convex domain $\Omega\subset\R^2$. As $u$ satisfies Neumann boundary conditions, its level set cross the border $\D\Omega$ orthogonally. Since the domain is convex, these level sets go apart one from each other as we move outward $\D\Omega$. As $\vert\nabla u\vert$ corresponds to the inverse of the distance of two level sets, it implies that $\vert\nabla u\vert$ decreases as we move outward $\Omega$, hence the result.

\section{Patterns in unbounded domains}\label{sec:UnboundedDomains}

\subsection{Non-degenerate patterns - proof of \autoref{thmStableNonDegenerate}}\label{Sec:ProofNonDegenerate}

The proof follows the same lines as the first two steps of section~\ref{Sec:FormalApproach}. But since $\Omega$ is unbounded, the computations which lead to "$\mathcal{F}(v_i)=0$" are not licit. We shall instead perform the computations on truncated functions.
We denote $v_i:= \D_{x_i} u$ for ${i\in\{1,\dots,n\}}$.
Differentiating~\eqref{ANG_Intro_EquationSemilineaire}, we find that all the $v_i$ satisfy
\begin{equation}\label{LinearEquation_Proof}
v_i(\Delta v_i+f'(u) v_i)\geq 0\quad \text{in } \Omega.
\end{equation}
We consider a cut-off function: for $R>0$, set
\begin{equation}\label{DefCutOff}
\chi_R(x):=\chi\left(\frac{\vert x\vert}{R}\right), \quad \forall x\in\R^n,
\end{equation} 
with $\chi$ a smooth nonnegative function such that $${\chi(z)=
\left\{\begin{aligned}
&1 &&\text{if }0\leq z\leq 1\\
&0 &&\text{if } z\geq 2
\end{aligned}\right.},\quad\quad \vert \chi'\vert\leq 2.$$
Multiplying~\eqref{LinearEquation_Proof} by $\chi_R^2$, integrating on $\Omega$, using the divergence theorem and~\eqref{BorderCondition2} we obtain
\begin{equation}
0\leq \int_{\D\Omega}\chi_R^2v_i\D_\nu v-\int_\Omega \nabla\left[\chi_R^2 v_i\right]\cdot\nabla v_i+\int_\Omega f'(u)\chi_R^2v_i^2.
\end{equation}
Using the identity $\left\vert\nabla\left[\chi_R v_i\right]\right\vert^2 = \nabla\left[\chi_R^2 v_i\right]\cdot\nabla v_i+\vert\nabla\chi_R\vert^2 v_i^2$ and rearranging the terms, we deduce 
\begin{equation}\label{Intermediate_Proof_2}
\lambda_1\int_\Omega \chi_R^2 v_i^2\leq \mathcal{F}\left(\chi_Rv_i\right)
\leq \frac{1}{2}\int_{\D\Omega} \chi_R^2\D_\nu v_i^2+\int_{\substack{\Omega}} \vert\nabla\chi_R\vert^2v_i^2.
\end{equation}
Setting $v=\vert \nabla u\vert$, we have $v^2=v_1^2+\dots+v_n$ and therefore summing the above inequality over $i\in\{1,\dots,n\}$ gives
\begin{equation*}
\lambda_1\int_\Omega \chi_R^2 v^2\leq\frac{1}{2}\int_{\D\Omega} \chi_R^2\D_\nu v^2+\int_{{\Omega}} \vert\nabla\chi_R\vert^2v^2.
\end{equation*}
Since $\Omega$ is convex, \autoref{LemmeIntermediaireVariations} implies that $\D_\nu v^2=\D_\nu \vert\nabla u\vert^2\leq0$ on $\D\Omega$. We deduce
\begin{equation*}
\lambda_1\leq \frac{\int_{{\Omega}} \vert\nabla\chi_R\vert^2v^2}{\int_\Omega \chi_R^2 v^2}
\leq \frac{\frac{4}{R^2}\int_{\Omega\cap\{\vert x\vert<2R\}}v^2}{\int_{\Omega\cap\{\vert x\vert<R\}} v^2}=:4\alpha_R.
\end{equation*}

We claim that
\begin{equation}\label{ClaimAlpha}
\liminf\limits_{R\to+\infty} \alpha_R\leq 0.
\end{equation}
If~\eqref{ClaimAlpha} holds, then $\la_1\leq0$, which completes the proof.
By contradiction, let us assume $\alpha_R\geq\delta>0.$ We use the notation $\mathcal{C}(R):=\int_{\Omega\cap\{\vert x\vert<R\}} v^2$ so that we can write $\alpha_R= \frac{\mathcal{C}(2R)}{R^2\mathcal{C}(R)}.$ The contradictory assumptions now reads $\mathcal{C}(2R)\geq \delta R^2\mathcal{C}(R)$. Iterating this inequality for a fixed $R>0$, we find
\begin{equation*}
 \mathcal{C}(2^{j}R)\geq K\left(\delta R^{2}\right)^j\quad\text{for all integer }j\geq1,
 \end{equation*} where, here and later, positive constants are generically denoted $K$. 
 In addition, $v$ is bounded, hence $\mathcal{C}(R)\leq KR^n$. We have
\begin{equation*}
 K\left(2^jR\right)^n\geq\left(\delta R^{2}\right)^j.
 \end{equation*} 
If $R$ is large enough, we reach a contradiction as $j$ goes to $+\infty$.
Thereby, we have proved~\eqref{ClaimAlpha}, and the proof of \autoref{thmStableNonDegenerate} is complete.

%
%
%
%

%
%
%

\subsection{A Liouville type result, or the simplicity of the principal eigenvalue}\label{Sec:schrodinger}
In this section, we adapt the step 3 of section~\ref{Sec:FormalApproach} to unbounded domains. Let us first introduce the principal eigenfunction $\varphi$, associated to $\lambda_1$.
The Euler-Lagrange equation associated with the functional $\mathcal{F}$ is obtained as a linearization of~\eqref{ANG_Intro_EquationSemilineaire} at $u$:
\begin{equation}\label{ANG_Intro_EquationLinarise}
    \left\{
    \begin{aligned}
        &-\Delta \psi-f'(u)\psi=0, &&\text{in }\Omega,\\\ 
        &\D_\nu \psi=0 &&\text{on }\D\Omega.
   \end{aligned}
    \right.
\end{equation}
If the domain $\Omega$ is bounded, then $\lambda_1$ is an eigenvalue of the linearized operator~\eqref{ANG_Intro_EquationLinarise}, called the \emph{principal eigenvalue}. When the domain is unbounded, we will refer to $\lambda_1$ as the \emph{generalized principal eigenvalue}.
It is associated with an eigenfunction $\varphi$ (\autoref{PropositionEigenfunctionVariation}) which is positive on $\overline\Omega$ and satisfies
\begin{equation}\label{EquationPrincipalEigenfunction}
\left\{
\begin{aligned}
&-\Delta\varphi-f'(u)\varphi=\lambda_1\varphi &&\text{in }\Omega,\\
&\D_\nu\varphi =0 &&\text{on }\D\Omega.
\end{aligned}
\right.
\end{equation}
The term \emph{generalized} comes from the fact that $\varphi$ may not belong to $H^1(\Omega)$. See Appendix~\ref{sec:GeneralizedPrincipalEigenvalue} for more details.

If the domain is bounded, then $\lambda_1$ is simple. This may no longer hold in unbounded domains.
However, the following lemma claims that, if the domain satisfies~\eqref{DomaineGeneral}, any bounded minimizer of~$\mathcal{F}$ is a multiple of~$\varphi$. It is a refinement of Theorem 1.7 in~\cite{Berestycki1997b}.

\begin{lemma}\label{schrodinger2} 
Let $\Omega\subset\R^n$ satisfy \eqref{DomaineGeneral} and let $u$ be a stable solution of \eqref{ANG_Intro_EquationSemilineaire}.
Let $v$ be a $C^1(\overline{\Omega})\cap L^{\infty}(\Omega)$ satisfy
\begin{equation}\label{LinearEquation}
v(\Delta v+f'(u) v)\geq 0\quad \text{in } \Omega.
\end{equation}
Consider the cut-off function introduced in~\eqref{DefCutOff} and assume that there exists a sequence of real positive numbers $(R_n)_{n\geq1}$ diverging to~$+\infty$ such that
\begin{equation}\label{BorderCondition2}
\int_{\D\Omega}\chi_{R_n}^2\D_\nu v^2\leq0,\qquad \forall n=1,2\dots
\end{equation}
Then $v\equiv C\varphi$ for some constant $C\in\R$, where $\varphi$ is a principal eigenfunction associated with $\lambda_1$.
\end{lemma}
\begin{proof}[\autoref{schrodinger2}]
We follow the method of~\cite{Berestycki1997b}. Let us set ${ \sigma=\frac{v}{\varphi}}$ and show that $\sigma$ is constant. From~\eqref{LinearEquation}, we deduce
\begin{equation*}
\sigma\varphi\left(\varphi\Delta \sigma+2\nabla\varphi\cdot\nabla\sigma+\sigma\left(\Delta\varphi+f'(u)\varphi\right)\right)\geq0,\qquad\text{a.e. in }\Omega.
\end{equation*}
From $\lambda_1\geq0$ and the equation satisfied by $\varphi$, we obtain
$${
\sigma\nabla\cdot(\varphi^2\nabla\sigma)\geq0,\qquad\text{a.e. in }\Omega.
}$$
Multiplying by $\chi_{R}^2$ (defined in \eqref{DefCutOff}), integrating on $\Omega$ and using the divergence theorem, we find
\begin{align*}
0&\leq \int_{\D\Omega}\chi_{R}^2\sigma\varphi^2\D_\nu\sigma-\int_\Omega\varphi^2\nabla\left(\chi_{R}^2\sigma\right)\cdot\nabla\sigma\\
&=\int_{\D\Omega}\chi_{R}^2\sigma\varphi^2\D_\nu\sigma
-\int_\Omega\varphi^2\chi_{R}^2\vert\nabla\sigma\vert^2
-2\int_\Omega\varphi^2\chi_{R}\sigma\nabla\chi_{R}\cdot\nabla\sigma.
\end{align*}
Since $\D_\nu\varphi=0$ on $\D\Omega$, the boundary term reads $\int_{\D\Omega}\chi_{R}^2v\D_\nu v$, which is nonpositive from \eqref{BorderCondition2} when $R=R_n$. 
From Cauchy-Schwarz inequality, we deduce
\begin{equation}\label{inequality}
\int_\Omega\chi_{R_n}^2\varphi^2\vert\nabla\sigma\vert^2
\leq 2\sqrt{\int_{\substack{\Omega_{2R_n}\backslash\Omega_{R_n}}}\chi_{R_n}^2\varphi^2\vert\nabla\sigma\vert^2}\sqrt{\int_\Omega v^2\vert\nabla\chi_{R_n}\vert^2},
\end{equation}
where $\Omega_R:=\Omega\cap\{\vert x\vert < R\}$.

Recalling that $\chi_R(x):=\chi\left(\frac{\vert x\vert}{R}\right)$ with $\vert \chi'\vert\leq 2$, assumption~\eqref{DomaineGeneral} implies
\begin{equation}\label{claimBounded}
\int_\Omega v^2\vert\nabla\chi_{R}\vert^2\text{ is bounded, uniformly in }R\geq1.
\end{equation}
From \eqref{inequality}, we deduce that
$\int_\Omega \chi_{R_n}^2\varphi^2\vert\nabla\sigma\vert^2$ is uniformly bounded. Using~\eqref{inequality} again, we infer that it converges to $0$ as $R_n\to\infty$. At the limit, we find
${
\int_\Omega\varphi^2\vert\nabla\sigma\vert^2\leq0.
}$
Hence $\nabla\sigma=0$, which ends the proof.
\end{proof}

The cornerstone of the proof is that $\sigma\nabla\cdot(\varphi^2\nabla\sigma)\geq0$ implies $\nabla \sigma=0$, where $\sigma:=\frac{v}{\varphi}$. The litterature refers to this property as a \emph{Liouville property}. Originally introduced in \cite{Berestycki1997b}, it has been extensively discussed (see \cite{Ambrosio2000,Barlow2000,Gazzola,Ghoussoub1998,Moschini2005}) and used to derive numerous results (e.g. \cite{Berestyckip,Cabrea,Cabre,Dancer2004,Alberti2001,Dupaigne2008}), in particular to prove the De Giorgi's conjecture in low dimensions.
\autoref{schrodinger2} is a refinement of this property for domains with a boundary, instead of $\Omega=\R^n$. This is why we need the boundary condition~\eqref{BorderCondition2}.

This is the only step where \eqref{DomaineGeneral} is needed, and it is thus a natural question to ask if this assumption can be relaxed. In the proof, \eqref{DomaineGeneral} is used to derive \eqref{claimBounded}, thus the choice of $\chi_R$ seems crucial. However, in~\cite{Ghoussoub1998}, the authors consider the optimal $\chi_R$ by taking the capacitary test function (see e.g.~\cite{Helms1969}), i.e., a solution of the minimization problem
\begin{equation}\label{minimization}
\inf\limits_{\chi\in H^1(\R^2)}\left[\int_{R\leq\vert x\vert\leq R'}\vert \nabla\chi(x)\vert^2\mathrm{d}x,\  \xi(x)=\left\{
\begin{aligned}
&1 &&\text{ if } \vert x\vert\leq R\\
&0 &&\text{ if } \vert x\vert\geq R'
\end{aligned} \right.\right].
\end{equation}
That, in fact, does not allow to substantially relax condition \eqref{DomaineGeneral}.
In \cite{Barlow1998}, Barlow uses a probabilistic approach to establish that the aforementioned Liouville property (and consequently \autoref{schrodinger2}) does not hold in $\Omega=\R^n$, $n\geq3$. 
It is thus reasonable to think that condition \eqref{DomaineGeneral} cannot be relaxed, yet this is an open question. We also cite~\cite{Gazzola}, in which \eqref{DomaineGeneral} is proved to be sharp, however, we point out that, there, the condition $v\in L^\infty$ is not satisfied.
Note also that, in the present work, we only apply \autoref{schrodinger2} to functions $v$, which are derivatives of $u$, which is a stronger condition than \eqref{LinearEquation}. In this context, not much is known about whether~\eqref{DomaineGeneral} could be relaxed. 
Indeed, up to the author's knowledge, the only available counterexample is for $\Omega=\R^n$, $n\geq7$~\cite{Pino2011,Pacard2013}.

However, we can sometimes relax~\eqref{DomaineGeneral} under further assumptions, either on $\Omega$, $f$, $v$, or $\varphi$. 
From a remark in \cite{Dupaigne2008}, if $f\geq 0$, we can relax~\eqref{DomaineGeneral} to $\sup_{R\to+\infty}\frac{\vert\Omega\cap\{\vert x\vert\leq R\}\vert}{R^4}<+\infty$. 
We can also adapt the arguments of \cite{Berestycki1997b,Cabrea,Moschini2005} to show that assumption~\eqref{DomaineGeneral} can be replaced by $v\in{H}^1(\Omega)$, or $v=o(\vert x\vert^{1-\frac{n}{2}})$, or $\inf_{\Omega}\varphi>0$.
In addition, we point out that \autoref{schrodinger2} holds for a large class of domains satisfying
\begin{equation}
\sup_{R\to+\infty}\frac{\left\vert \Omega\cap\{\vert x\vert\leq R\}\right\vert}{R^2\ln R}<+\infty.
\end{equation}
Namely, let $\Omega$ be of the form
\begin{equation}
\Omega:=\left\{(x,x')\in\R^2\times\R^n: x'\in\omega(x)\right\},
\end{equation}
where, $\forall x\in\R^n$, $\omega(x)\subset\R^n$ is bounded and
\begin{equation}
\sup_{\vert x\vert\to+\infty}\frac{\left\vert\omega(x)\right\vert}{\ln\vert x\vert}<+\infty.
\end{equation}
Then, to show~\eqref{claimBounded}, we use the cut-off
\begin{equation}
\chi_{R}(x)=\left\{\begin{aligned}
&1 &&\text{if }\vert x\vert\leq R,\\
&\frac{\ln R^2-\ln{\vert x\vert}}{\ln{R^2}-\ln{R}} &&\text{if }R\leq \vert x\vert\leq R^2,\\
&0 &&\text{if } \vert x\vert \geq R^2.
\end{aligned}\right.
\end{equation}
This cut-off was first introduced in~\cite{Ghoussoub1998} as a solution of~\eqref{minimization} for $n=2$.

\subsection{Non-existence of patterns - proof of \autoref{thmdimension2}}\label{Sec:ProofConvex}

We prove~\autoref{thmdimension2}.
Let $u$ be a stable solution of~\eqref{ANG_Intro_EquationSemilineaire} and consider $\varphi$ a principal eigenfunction associated with $\lambda_1$.
As a consequence of \autoref{schrodinger2} and \autoref{LemmeIntermediaireVariations}, we have the following intermediate result.
\begin{lemma}\label{keylemma}
For any $\xi\in\R^n$, $\nabla u\cdot \xi\equiv C_\xi\varphi$ for some constant $C_\xi$.
\end{lemma}
\begin{proof}
We assume without loss of generality that $\vert \xi\vert=1$, and that $\xi$ coincide with $e_1$ for $(e_1,\dots,e_n)$ an orthonormal basis of $\R^n$. We set $v_i:=\D_{x_i}u$ for $i\in\{1,\cdots,n\}$. Differentiating \eqref{ANG_Intro_EquationSemilineaire}, we find that $v_i$ satisfies~\eqref{LinearEquation}. Moreover, we know by assumption that the $v_i$ are bounded. 

Now, we show that all the $v_i$ satisfy~\eqref{BorderCondition2}.
On the one hand, since $\Omega$ is convex, \autoref{LemmeIntermediaireVariations} implies
\begin{equation*}
\sum\limits_{i=1}^n\int_{\D\Omega}\chi_R^2\D_\nu v_i^2=\int_{\D\Omega}\chi_R^2 \D_\nu\vert\nabla u\vert^2\leq 0.
\end{equation*}
On the other hand, \autoref{schrodinger2} implies in particular that for any $i\in\{1,\dots,n\}$ we have $\int_{\D\Omega}\chi_R^2 \D_\nu v_i^2\geq0$ for all $R\gg 1$, i.e., for $R$ large enough, all the terms of the above sum are nonnegative. As the sum is nonpositive, all the terms must be zero. 
Then, we apply \autoref{schrodinger2} to conclude.
\end{proof}

Let us complete the proof of \autoref{thmdimension2}.
The mapping
\begin{equation}
\begin{array}{ccc}
\R^n & \to & \R \\ 
\xi & \mapsto & C_\xi
\end{array} 
\end{equation}
is linear and therefore vanishes on a hyperplane $H$. Setting $e\in\mathbb{S}^{n-1}$ the unit vector orthogonal to $H$, we have $\nabla u=\vert\nabla u\vert e$, hence $u$ varies only in the direction $e$ (in other words, $u$ is a function of only one scalar variable). 

If $\Omega$ is not straight in the direction $e$, there exists a point $x_0\in\D\Omega$ on which the outer normal derivative is not colinear with $e$. From $\D_\nu u=0$ on $\D\Omega$, we deduce $\D_{e} u(x_0)=0$. From $\varphi>0$ on $\overline\Omega$, we deduce $\D_{e} u \equiv0$, thus $u$ is constant.

If $\Omega$ is straight in the direction $e$, then $\D_{e} u$ may be a nonzero multiple of $\varphi$. To fix ideas, we assume that $e$ corresponds to the $x_1$ direction.
Since $v_1\equiv C_1\varphi$, it is of constant sign, hence $u$ is flat and monotic.

Let us now assume that $u$ is bounded. First, note that, since $u$ is monotonic, it has a limit $z^+$ when $x_1\to+\infty$. Setting $u_n(x_1)=u(x_1+n)$ and using classical elliptic estimates, we can extract a subsequence that $C^2_{loc}$-converges to a stable solution $u_\infty$ of \eqref{ANG_Intro_EquationSemilineaire} (See section~\ref{sec:AsymptoticFormulation} for more details. Note also that $\Omega$ is invariant under translation in the $x_1$ direction). From $u_\infty\equiv z^+$, we deduce that $z^+$ must be a stable root of $f$. Identically, when $x_1\to-\infty$, $u$ converges to a stable root of $f$, denoted $z^-$. 

If $z^+=z^-$, then $u$ is constant. Let us assume $z^-\neq z^+$,  and fix $M>0$. Multiplying $-u''=f(u)$  by $u'$ and integrating on ${x_1\in[-M,M]}$ gives
\begin{equation}
\frac{1}{2}\left(u'(-M)^2-u'(M)^2\right)=\int_{u(-M)}^{u(M)} f.
\end{equation}
As $u'(\pm\infty)=0$ (indeed, $u'$ is integrable and $u''$ is bounded), when $M$ goes to $+\infty$ we obtain $\int_{z^-}^{z^+}f=0$. The proof of \autoref{thmdimension2} is thereby complete.

\subsection{Further symmetries}\label{Sec:SymmetryProperties}

In this section, we establish further symmetries for patterns in unbounded domains. Namely, we shall see that patterns inherit from the domain's invariance with respect to translations and planar rotations. 
The results of this section are mostly classical, except for some minor modifications. We provide some proofs for completeness.

\paragraph*{}
The following proposition deals with cylinders, possibly not convex, which are straight in some directions. This result is somehow already contained in~\cite{Ghoussoub1998,Dancer2004,Farina2013a,Farina2013b}.
\begin{proposition}\label{thmstraight2}
Let $\Omega=\R^n\times\omega$ with $\omega\subset\R^m$ bounded and let $u$ be a stable solution of~\eqref{ANG_Intro_EquationSemilineaire}. For $x\in\Omega$, we generically denote $x=(x_1,\dots,x_n,x'_1,\dots,x'_m)$. 
\begin{enumerate}
\item  If $u$ is stable non-degenerate, then $u$ does not depend on $(x_1,\dots,x_n)$.
\item If $u$ is stable and $n=1$, then $u$ is monotonic with respect to $x_1$.
\item If $u$ is stable and $n=2$, the dependance of $u$ with respect to $(x_1,x_2)$ occurs through a single scalar variable $x_0\in\R$. Moreover, $u$ is monotonic with respect to $x_0$.
\end{enumerate}
\end{proposition}

\begin{proof}[\autoref{thmstraight2}]
We set $v_i:= \D_{x_i} u$ for ${i\in\{1,\dots,n\}}$, which are bounded by assumption. Differentiating~\eqref{ANG_Intro_EquationSemilineaire}, we find that $v_i$ satisfies~\eqref{LinearEquation_Proof}. We can then proceed as in the proof of \autoref{thmStableNonDegenerate} to show that~\eqref{Intermediate_Proof_2} holds, namely
\begin{equation}
\lambda_1\int_\Omega \chi_R^2 v_i^2\leq \mathcal{F}\left(\chi_Rv_i\right)
\leq \frac{1}{2}\int_{\D\Omega} \chi_R^2\D_\nu v_i^2+\int_{\Omega} \vert\nabla\chi_R\vert^2v_i^2,
\end{equation}
where $\chi_R$ is the cut-off function introducedin~\eqref{DefCutOff}. Since $\Omega$ is straight in the directions $x_i$, we have $\D_\nu v_i=0$ on $\D\Omega$. Therefore we have
$$
\lambda_1\int_\Omega \chi_R^2 v_i^2\leq \int_{\Omega} \vert\nabla\chi_R\vert^2v_i^2
$$
and we can conclude as in the proof of \autoref{thmStableNonDegenerate} (in particular using~\eqref{ClaimAlpha}) that if $v_i\not \equiv0$ then $\lambda_1\leq0$. It proves the first astatement.

Next, if $n\leq 2$ then the domain satisfies~\eqref{DomaineGeneral}. Since $v_i$ satisfies~\eqref{LinearEquation} and \eqref{BorderCondition2}, we deduce from \autoref{schrodinger2} that $v_i$ is a multiple of $\varphi$. The second assertion follows from the fact that $\varphi>0$. To prove the last statement, note that, as in the proof of \autoref{thmdimension2}, there exists a direction $\xi\in\mathbb{S}^{n-1}$ for which $\nabla u\cdot \xi\equiv0$.
\end{proof}

\paragraph*{}
We are now interested in cylinders which are invariant with respect to a planar rotation and shall see that patterns inherit this symmetry.
This property was initially stated in~\cite{Matano1979} for bounded domains, and has been extended to manifolds~\cite{Lopes1996a,Jimbo1984,Rubinstein1994,Bandle2012,Punzo2013,Farina2013a,Farina2013b}.

\begin{definition}\label{DefinitionThetaInvariant}
A domain $\Omega\subset\R^{n+2}$ is said to be $\theta$-invariant if $\Omega=\Omega'\times[0,2\pi)$, where $\Omega'\subset\R^n\times\R^+$ in some cylindrical coordinates $(x,r,\theta)\in\R^n\times\R^+\times[0,2\pi)$.
\end{definition}

\noindent When considering a $\theta$-invariant domain $\Omega$, we further assume that the radial section is uniformly bounded:
\begin{equation}\label{AssumptionRBounded}
\sup\limits_{(x,r)\in\Omega'} r<+\infty.
\end{equation}
In particular, it guarantees that if $u$ is a solution of \eqref{ANG_Intro_EquationSemilineaire} in $\Omega$, then $\D_\theta u$ is bounded. Note that this assumption is strictly needed, since planar patterns may exist in $\R^2$.

\begin{proposition}\label{thmtheta} 
Let $\Omega$ be a $\theta$-invariant domain which satisfies~\eqref{AssumptionRBounded} and $u$ be a stable solution of~\eqref{ANG_Intro_EquationSemilineaire}. Assume that~\eqref{DomaineGeneral} holds or that $u$ is stable non-degenerate. Then $\D_\theta u=0$. \newline 
\end{proposition}
\begin{proof}
We set $v:=\D_\theta u$. Note that, in a Cartesian system of coordinates $$z=(x_1,\dots,x_n,y_1,y_2)\in\R^{n+2},$$ we have $v(z)=y_2\D_{y_1}u(z)-y_1\D_{y_2}u(z).$ We also know from~\eqref{AssumptionRBounded} that $v$ is bounded.
Differentiating~\eqref{ANG_Intro_EquationSemilineaire}, we find that $v$ satisfies~\eqref{LinearEquation_Proof}.
Moreover, as $\Omega$ is $\theta$-invariant, we have $\D_\nu \D_\theta u=0$ on $\D\Omega.$ We can therefore proceed as in the proof of~\autoref{thmStableNonDegenerate} to prove that if $v\not\equiv0$ then $\lambda_1\leq0$. It completes the proof for the case where $u$ is stable non-degenerate.

Assume now that~\eqref{DomaineGeneral} holds. Since $v$ satisfies~\eqref{LinearEquation} and~\eqref{BorderCondition2}, \autoref{schrodinger2} implies ${v\equiv C\varphi}$ for some constant $C$. Thus, $v$ is of constant sign. For $(x,x',r,\theta)\in\Omega$ we have
\begin{equation*}
\int_0^{2\pi} v(x,x',r,\theta)d\theta=0.
\end{equation*} 
Therefore $v\equiv 0$, which completes the proof.
\end{proof}

\begin{remark}
As noticed by Matano~\cite{Matano1979}, since $\D_\theta u=0$, then $w:=\D_r u$ satisfies 
\begin{equation*}
w(\Delta w+f'(u)w)=\frac{1}{r^2}w^2\quad\text{a.e in }\Omega,
\end{equation*}
and $w$ satisfies~\eqref{LinearEquation}. Therefore, if we further assume that $\Omega'$ is convex, then \autoref{thmStableNonDegenerate} and \autoref{thmdimension2} apply. It proves the non-existence of patterns in some non-convex domains, such as rings or torus.

Similar strategies and refined results can be found in the work of Alikakos and Bates~\cite{Alikakos1988}. In this work, the authors investigate in particular the existence of patterns when the equation features a radial source term, the domain is a ball or an annulus, and the reaction term is multiplied by a large factor $\eps^{-1}$ with $\eps\ll1$.
\end{remark}

\section{Asymptotic symmetries}\label{sec:AsymptoticFormulation}
In this section, we establish asymptotic symmetries for patterns. We prove \autoref{th:AsymptoticNonDegenerate} and also the following analogous result for possibly degenerate patterns.
\begin{theorem}\label{th:AsymptoticDegenerate}
Under the same assumptions as in \autoref{th:AsymptoticNonDegenerate}, but with $u$ a bounded pattern which is possibly degenerate. Suppose that the stable roots of $f$, denoted $(z_i)$, are isolated, and that $\int_{z_i}^{z_j} f\neq 0$, for $i\neq j$. Let us further assume that the limiting domain $\Omega_\infty$ satisfies~\eqref{DomaineGeneral}.
Then, $u(x_1,\cdot)$ converges $C^2_{loc}$ to a stable root of $f$ when $x_1\to+\infty$. 
\end{theorem}
In section~\ref{sec:Symmetry}, we also adapt the symmetry results of~\ref{sec:Symmetry} when the domain is {asymptotically} invariant under a translation or a rotation.

First, let us give a precise definition of the convergence of a domain.
\begin{definition}\label{DefConvergenceDomain}
Let $\Omega\subset\R^n$ be a uniformly smooth domain. For any $y\in\R$, we define the translated domains (on the $x_1$-axis)
$$\Omega[y]:=\left\{(x_1,\dots,x_{n}): (x_1+y,\dots,x_{n})\in\Omega\right\}.$$
We say that $\Omega$ converges to $\Omega_\infty\subset\R^n$ if the boundary of $\Omega[y]$ converges to that of $\Omega_\infty$ when $y\to+\infty$ in the $C^{2,\alpha}_{loc}$ topology.
\end{definition}

\subsection{Convergence to a constant - proof of \autoref{th:AsymptoticNonDegenerate} and \autoref{th:AsymptoticDegenerate}} 

Assume $\Omega\subset\R^{n}$ is uniformly smooth and converges to a convex domain $\Omega_\infty\subset\R^n$. Let $u$ be a solution of~\eqref{ANG_Intro_EquationSemilineaire}.
For technical reasons, we need to extend $u$ in all $\R^{n}$. 
\begin{lemma}\label{LemmaExtension}
We can extend $u$ to a uniformly $C^{2,\alpha}$ function (still denoted $u$) which is defined in $\R^{n}$, coincides with $u$ on $\overline\Omega$, and is identically $0$ in $\R^{n}\backslash U$, $U$ being a neighborhood of $\overline\Omega$.
\end{lemma}
\begin{proof}
As $u$ satisfies Neumann boundary conditions, we can extend it by reflexion on an open set $U\supset\overline\Omega$. Note that, since the domain is uniformly $C^{2,\alpha}$, we can choose $U$ such that
\begin{equation}
\inf\limits_{(x,y)\in\Omega\times\D U}\vert x-y\vert>0.
\end{equation}
The extended function, denoted $\tilde u$, satsfies an elliptic equation on $U$ for which classical global $C^{2,\alpha}$ estimate hold (see e.g. Theorem 6.30 in~\cite{GilbarDavid2015}). This procedure is classical but technical, see for instance Appendix A in~\cite{Raspail2000}. From this, we infer global $C^{2,\alpha}$ estimates for $\tilde u$ in $\overline\Omega$.
Finally, choosing 
any open set $\tilde U$ such that $\Omega\subset\tilde U\subset U$, we can define $\doubletilde u\in C^{2,\alpha}(\R^n\backslash\Omega)$ which coincides with $\tilde u$ in $\tilde U$ and is identically $0$ in $\R^{N}\backslash U$.
\end{proof}

We consider the $\omega$-limit set
\begin{equation}\label{DefinitionOmegaLimitSet}
\Sigma_u:=\bigcap\limits_{y\in\R}\overline{\left\{u[y']:y'\geq y\right\}}\subset L_{{loc}}^\infty(\R^{N}),
\end{equation}
where 
$$u[y]: (x_1,\dots,x_{N})\mapsto u(x_1+y,\dots,x_{N}),\qquad \forall y\in\R.$$
The topological closure should be understood in the $L_{{loc}}^\infty$ sense.
The key point is the following observation, which states that a solution is ``more stable at infinity''.
\begin{lemma}\label{lambda_infty}
Let $u_\infty\in\Sigma_u$. It is a solution of \eqref{ANG_Intro_EquationSemilineaire} in $\Omega_\infty$.
Moreover, $
\lambda_1(u,\Omega)\leq\lambda_1(u_\infty,\Omega_\infty).
$
\end{lemma}
\begin{proof}
Let $u_\infty\in\Sigma_u$. There exists a sequence $y_n\to+\infty$ such that $u[y_n]\to u_\infty$. As a consequence of the locally uniform $C^{2,\alpha}$ estimates in $\overline{\Omega}$ (see the proof of \autoref{LemmaExtension}), we deduce that the convergence $u[y_n]\to u_\infty$ occurs in $C_{{loc}}^2$. Thus, $u_\infty$ is a solution of~\eqref{ANG_Intro_EquationSemilineaire} in $\Omega_\infty$.

Consider $\varphi\in C^{2}(\overline\Omega)$ a principal eigenfunction associated with $\lambda_1(u,\Omega)$ (given by \autoref{PropositionEigenfunctionVariation}), define $a_n:=\varphi[y_n](0)$ and $\varphi_n:=\frac{1}{a_n}\varphi[y_n]$. Note that $\lambda_1(u[y_n],\Omega[y_n])=\lambda_1(u,\Omega)$. We have
\begin{equation*}
\left\{\begin{aligned}
&-\Delta\varphi_n-f'(u[y_n])\varphi_n=\lambda_1(u,\Omega)\varphi_n&\text{in }\Omega[y_n],\\
&\D_\nu\varphi_n=0&\text{on }\D\Omega[y_n].
\end{aligned}\right.
\end{equation*}
We can extend by reflexion $\varphi_n$ on a neighborhood of $\overline{\Omega[y_n]}$ on which it satisfies an elliptic equation, thus can apply the Harnack inequality in $\overline{\Omega_\infty}$ to infer that $\varphi_n$ is locally bounded, uniformly in $n$ (for more details, see the proof of Proposition 1 p.30 in \cite{Raspail2000}). Thanks to classical elliptic estimates, we can extract a subsequence (still denoted $n$) such that $\varphi_n$ converges in $C^2_{loc}(\overline\Omega_\infty)$ to some $\varphi_\infty$. Then, $\varphi_\infty>0$ in $\Omega_\infty$ and $(\Delta+f'(u_\infty)+\lambda_1(u,\Omega))\varphi_\infty\leq0$. From \autoref{eigenfunction}, we deduce $\lambda_1(u_\infty,\Omega_\infty)\geq \lambda_1(u,\Omega)$.
\end{proof}

We also need the following general lemma, which is adapted from a classical result (see, for example, Theorem 2.9 in \cite{Matano1979}).
\begin{lemma}\label{LemmaConnected}
$\Sigma_u$ is a connected set in the $L^\infty_{{loc}}(\R^n)$ topology.
\end{lemma}
\begin{proof}
By contradiction, assume there exists $Y\subset\Sigma_u$ both open and closed, $Y\neq\emptyset$ and $Y\neq\Sigma_u$. 
Note that $\Sigma_u$ is a compact subset of $L^\infty_{{loc}}(\R^n)$, thus $Y$ is compact. As, in addition, $\Sigma_u\backslash Y$ is a closed set, there exists a an open set $V\subset L^\infty_{{loc}}(\R^n)$ and a closed set $F\subset L^\infty_{{loc}}(\R^n)$ such that $Y\varsubsetneq V\varsubsetneq F$ and $Y=\Sigma_u\cap F$.

Since $Y\neq\emptyset$ and $Y\neq \Sigma_u$, there exists a real sequence $\left(y_n\right)_{n\geq0}$ such that $y_n\to+\infty$ and, for all integer $n\geq0$,
\begin{equation*}
u[y_{2n}]\in V,\quad
\quad u[y_{2n+1}]\not\in F.
\end{equation*}
By continuity of the mapping $y\mapsto u[y]$, we deduce that, for all $n\geq0$, there exists $\tilde y_n\in[y_{2n},y_{2n+1}]$ such that
\begin{equation}\label{ContradictionAsymptotics}
u[\tilde y_n] \in F\backslash V.
\end{equation}
As the sequence $u[\tilde y_n]$ is uniformly $C^{2,\alpha}$, it converges up to an extraction to some $u_\infty\in\Sigma_u$. 
But we also have $u_\infty\in F\backslash V$, thus $u_\infty\in Y$ and $u_\infty\not\in Y$: contradiction.
\end{proof}

We are now ready to prove the main results of this section.
\begin{proof}[\autoref{th:AsymptoticNonDegenerate} and \autoref{th:AsymptoticDegenerate}]
Let $u_\infty\in\Sigma_u$. From \autoref{lambda_infty}, $u_\infty$ is a stable solution of \eqref{ANG_Intro_EquationSemilineaire} in  $\Omega_\infty$, which is convex.
From \autoref{thmStableNonDegenerate} and \autoref{thmdimension2}, we deduce that $u_\infty$ is constant. We have
$$\Sigma_u\subset \mathcal{Z}:=\left\{z\in\R: \quad f(z)=0,\quad f'(z)\leq0\right\}.$$

Assume $f$ has only isolated zeros. Then $\Sigma_u$ is a discrete set, and is also connected (\autoref{LemmaConnected}). Hence, it is a singleton, which achieves the proof.

Assume instead that $u$ is stable non-degenerate, then we have
$$\Sigma_u\subset \mathcal{Z}^\star:=\left\{z\in\mathcal{Z}:f'(z)<0\right\},$$
which is a discrete set, and we conclude as above.
\end{proof}

\begin{remark}\label{Rmk:StabilityOutsideCompact}
The proof of \autoref{th:AsymptoticNonDegenerate} and \autoref{th:AsymptoticDegenerate} and also the proof of the other asymptotic symmetry properties presented in the next section can be adapted to the case where $u$ is only assumed to be stable outside a compact $\mathcal{K}\subset\overline{\Omega}$, i.e., when the condition on the sign of $\lambda_1$ is replaced by a condition on the sign of
\begin{equation}
\lambda_{1,\mathcal{K}}= \inf\limits_{\substack{\psi\in C^1_0(\overline{\Omega}\setminus\mathcal{K})\\ \Vert \psi\Vert_{{L}^2}=1}}\int_{\Omega\setminus\mathcal{K}}\vert\nabla \psi\vert^2-f'(u)\psi^2.
\end{equation}
To adapt the proofs to this case, simply note that if $\Omega$ converges to $\Omega_\infty$, then so does $\Omega\setminus\mathcal{K}$.
\end{remark}

\subsection{Further asymptotic symmetries}\label{sec:Symmetry}
In this section, we establish further asymptotic symmetries for patterns when the domain is asymptotically straight in one direction or invariant with respect to a planar rotation. We define what we mean for a domain to satisfy a geometrical property \emph{asymptotically}. Note that we allow the domain not to converge to a limiting domain (in the sense of \autoref{DefConvergenceDomain}).

\begin{definition}
Let $\Omega\subset\R^n$ be a uniformly $C^{2,\alpha}$ domain. For any real sequence $y_n\to+\infty$, $\Omega[y_n]$ converges (up to an extraction) to some $\Omega_\infty\subset\R^n$ (in the sense of \autoref{DefConvergenceDomain}).
We define the set of all possible limiting domains 
$$\Gamma_\Omega:=\left\{\Omega_\infty\subset\R^n: \exists y_n\to+\infty,\ \Omega[y_n]\text{ converges to }\Omega_\infty\right\}.$$
We say that a domain $\Omega$ satisfies a geometrical property asymptotically if every $\Omega_\infty\in\Gamma_\Omega$ satisfies this property.
\end{definition}

\begin{proposition}\label{PropositionStraight}
Assume that $\Omega$ is asymptotically straight in a direction $e\in\mathbb{S}^N$. Let $u$ be a stable solution of \eqref{ANG_Intro_EquationSemilineaire}. If $u$ is stable non-degenerate, then $\D_eu$ converges to $0$, $C^1_{loc}$-uniformly when $x_1\to+\infty$.
\end{proposition}
\begin{proof}
Let $y_n\to+\infty$. Up to an extraction (still denoted $y_n$), $u[y_n]$ converges to some $u_\infty\in\Sigma_u$ in $C^{2,\alpha}_{{loc}}$ and $\Omega[y_n]$ converges to some $\Omega_\infty\in\Gamma_\Omega$. From \autoref{lambda_infty}, we deduce that $u_\infty$ is a stable non-degenerate solution of~\eqref{ANG_Intro_EquationSemilineaire} in $\Omega_\infty$. As $\Omega_\infty$ is straight in the direction $e$,  using \autoref{thmstraight2} we deduce that $\D_e u_\infty\equiv0$, hence the result.
\end{proof}

\autoref{th:AsymptoticCoro} follows from the previous result in the particular case when the direction $e$ coincides with that of $x_1$.
\begin{proof}[\autoref{th:AsymptoticCoro}]
We apply the same method as in the proof of \autoref{th:AsymptoticNonDegenerate}, using \autoref{thmstraight2} and the fact that stable non-degenerate solutions are isolated among solutions (\autoref{LemmaIsolated}).
\end{proof}
Note that, if we assume that the set of stable solutions of~\eqref{EquationInSection} is discrete, then \autoref{th:AsymptoticCoro} extends to possibly degenerate patterns.

We now turn to the case of a domain which is asymptotically invariant with respect to a planar rotation.
\begin{proposition}
Assume $\Omega$ is asymptotically $\theta$-invariant (\autoref{DefinitionThetaInvariant}) and satisfies
\begin{equation}
\sup\limits_{(x,r,\theta)\in\Omega} r<+\infty.
\end{equation}
Let $u$ be a stable solution of~\eqref{ANG_Intro_EquationSemilineaire}. Assume either that $u$ is stable non-degenerate, or that $\Omega$ satisfies~\eqref{DomaineGeneral}. Then $\D_{\theta}u\to0$, $C^1_{loc}$-uniformly when $x_1\to+\infty$.
\end{proposition}
\begin{proof}
We proceed as in the proof of \autoref{PropositionStraight} and we use \autoref{thmtheta} instead of \autoref{thmstraight2}.
\end{proof}

\appendix

\section{Generalized principal eigenvalue}\label{sec:GeneralizedPrincipalEigenvalue}
This section is devoted to defining the \emph{generalized principal eigenvalue} of a linear operator, and to state some properties. The term \emph{generalized} is used when dealing with unbounded domains, in which there may not exist eigenfunctions in $H^1$. Here, we focus on the essential aspects and omit the details: the content of this section will be developed in a forthcoming paper~\cite{Nordmann2019}

\subsection{Definition}

We generally consider a smooth domain $\Omega$ and a linear elliptic operator
\begin{equation}\label{DefinitionEllipticOperator}
\mathcal{L}u(x):=\mathrm{div}\left(A(x)\cdot\nabla u(x)\right)+B(x)\cdot\nabla u(x)+c(x)u(x),\quad\forall x\in\Omega,
\end{equation}
where, $c:\Omega\to\R$, $B:\Omega\to\in\R^n$, and $A: \Omega \to \R^{n\times n}$ such that $A(x)$ is positive-definite (uniformly in $x\in\Omega$). For simplicity, we assume that the coefficients are smooth.
We associate the operator $\mathcal{L}$ with Neumann boundary conditions
\begin{equation}
\mathcal{B} u(x):=\D_{\nu_A} u(x)=0,\quad\forall x\in\D\Omega,
\end{equation}
with $\nu$ the outer normal derivative and $\D_{\nu_A} u:=\nu\cdot A\cdot \nabla u$ the co-normal outer derivative of $u$ associated with $A$. We focus here on Neumann boundary conditions, but we keep the notation $\mathcal{B}$ to emphasize that our statements can be adapted to other boundary conditions.

We consider the following eigenproblem:
\begin{equation}\label{EigenvalueRobin}
\left\{\begin{aligned}
&-\mathcal{L}\psi=\lambda\psi &&\text{in }\Omega,\\
&\mathcal{B} \psi=0 &&\text{on }\D\Omega.
\end{aligned}\right.
\end{equation}
If the domain is bounded, the Krein-Rutman theory gives the existence of an eigenvalue $\lambda_1$ to~\eqref{EigenvalueRobin}, called the \emph{principal eigenvalue}. This eigenvalue is real and minimizes the real part of the spectrum. In addition, $\lambda_1$ is simple, and is the only eigenvalue associated with a positive eigenfunction (called \emph{principal eigenfunction}). We let the reader refer to~\cite{Protter1984,GilbarDavid2015} for more details.
We point out that a fundamental property is that the validity of the Maximum Principle for the operator $(\mathcal{L},\mathcal{B})$ is equivalent to the condition $\lambda_1>0$.

If the domain is unbounded, Krein-Rutman's theory cannot be applied because the elliptic operator does not have compact resolvents. However, we can still define the notion of principal eigenvalue.
Following the approach of~\cite{Berestycki1994,Berestycki2015b,Rossi2020}, we give the following definitions.
\begin{definition}
\begin{itemize}
\item A function $u \in C^{2}(\overline\Omega)$ is said to be a subsolution (resp. supersolution) if
\begin{equation}
\left\{\begin{aligned}
&-\mathcal{L}u\leq 0\  \text{(resp. $\geq0$)}&&\text{in }\Omega,\\
&\mathcal{B}u\leq0\  \text{(resp. $\geq0$)}&&\text{on }\D\Omega.
\end{aligned}\right.
\end{equation}
\item We define the \emph{generalized principal eigenvalue} of $(\mathcal{L},\mathcal{B})$ as
\begin{equation}\label{DefinitionLambda1Robin}
\lambda_1:=\sup\left\{\lambda\in\R: (\mathcal{L}+\lambda,\mathcal{B})\text{ admits a positive supersolution}
 \right\}.
\end{equation}
\end{itemize}
\end{definition}
This definition coincides with the classical definition in the case of a bounded domain, and coincide with the definition~\eqref{DefLambda} when $\mathcal{L}$ is self-adjoint. We are about to see that $\lambda_1$ admits a positive eigenfunction. However, $\lambda_1$ may not be simple.

\subsection{Existence of a positive eigenfunction}

A remarkable property is that, even in unbounded domains, $\lambda_1$ is associated with a positive eigenfunction.
\begin{proposition}\label{PropositionEigenfunctionVariation}
Let $\Omega\subset\R^n$ be a smooth (possibly unbounded) domain and $\mathcal L$ an elliptic operator as in~\eqref{DefinitionEllipticOperator}. 
There exists $\varphi\in C^2(\overline{\Omega})$ which is positive on $\overline{\Omega}$ and satisfies
\begin{equation}\label{EquationGeneralizedEigenfunction}
\left\{\begin{aligned}
&-\mathcal{L}\varphi=\lambda_1 \varphi&\text{in }\Omega,\\
&\mathcal{B} \varphi=0&\text{on }\D\Omega.
\end{aligned}\right.
\end{equation}
We refer to $\varphi$ as a {principal eigenfunction} of~$(\mathcal{L},\mathcal{B})$.
\end{proposition}
\begin{proof}
The proof follows closely~\cite{Berestycki2015b} (see the proofs of Theorem 3.1 and Proposition 1).
For any $R>0$, let $D_R\subset\R^n$ be a smooth connected open set such that $B_R\subset D_R\subset B_{2R}$, with $B_R$ the ball of radius $R$. We also choose $D_R$ to be increasing with $R$, and to be such that $\D D_R\cap\D \Omega$ is a $C^2$ $(n-2)$-dimensional manifold. Set $\Omega_R:=\Omega\cap D_R$ and consider the eigenvalue problem with mixed boundary conditions
\begin{equation}\label{EigenvalueRobinTruncated}
\begin{aligned}
&-\mathcal{L}\psi=\lambda\psi&&\text{a.e. in }\Omega_R,\\
&\mathcal{B} \psi=0&&\text{a.e. on }\D\Omega\cap D_R,\\
&\psi=0&&\text{a.e on }\Omega\cap \D D_R.
\end{aligned}
\end{equation}
From the results of Liberman~\cite{Lieberman1986}, we know that all classical results (Schauder estimate, Maximum Principle, solvability, etc.) hold from the mixed boundary value problem above. As $\Omega_R$ is bounded, the weak Krein-Rutman theorem (e.g., Corollary 2.2 in~\cite{Nussbaum1981}) provides a pair of principal eigenelements $(\lambda_1^{R},\varphi^{R})$, where $\varphi^R\in W^{2,n}$. From Hopf's lemma, we have $\varphi>0$ on $\overline{\Omega}$. We choose the normalization $ \varphi(0)=1$.
Note that we impose Dirichlet boundary conditions on $\Omega\cap \D D_R$ to ensure the decreasing monotonicity of $R\mapsto\lambda_1^{R}$. Hence, $\lambda_1^{R}$ converge to some $\underline \lambda_1$ when $R\to+\infty$. 

Now, fix a compact $0\in K\subset\overline\Omega$ and assume that $R$ is large enough so that $K\subset\overline{\Omega}_R\backslash D_R$. From Theorem~3.3 in~\cite{Lieberman1987} and Theorem~4.3 in~\cite{Lieberman2001}, we derive a Harnack estimate, that is, 
\begin{equation}
\sup_{K} \varphi^R\leq C\inf_K \varphi^R
\end{equation}
with a constant $C$ independent of $R$. From $\varphi(0)=1$, we deduce that $\varphi^R$ is bounded in $K$, uniformly in $R$. From classical Schauder estimates, we deduce that $\varphi^R$ is $C^{2,\alpha}(K)$, uniformly in $R$.
Up to extraction of a subsequence, $\varphi^{R}$ converges to some $\varphi$ in $C^2(K)$. From a diagonal argument, we are provided with $\varphi\in C^2(\overline{\Omega})$ which satisfies
\begin{equation}
\left\{\begin{aligned}
&-\mathcal{L}\varphi=\underline\lambda_1 \varphi&\text{in }\Omega,\\
&\mathcal{B} \varphi=0&\text{on }\D\Omega,
\end{aligned}\right.
\end{equation}
and $\varphi>0$ on $\overline{\Omega}$.
Consequently $\lambda_1=\underline{\lambda_1}$, which achieves the proof.
\end{proof}
As a direct consequence,
\begin{corollary}\label{eigenfunction}
There exists a positive supersolution of $(\mathcal{L},\mathcal{B})$ in $\Omega$ if and only if $\lambda_1\geq0$.
\end{corollary}

\subsection{The Rayleigh-Ritz variational formula}\label{sec:self-adjoint_RayleighQuotient}
In the self-adjoint case, i.e., if $B\equiv0$ in~\eqref{DefinitionEllipticOperator}, the principal eigenvalue can be expressed through the Rayleigh-Ritz variational formula. This result is classical in bounded domains.
\begin{proposition}\label{th:RayleighFormula}
Assume $\Omega$ is smooth (possibly unbounded) and that $\mathcal{L}$ is a self-adjoint elliptic operator. For $\lambda_1$ defined in~\eqref{DefinitionLambda1Robin}, we have
\begin{equation}\label{RayleighFormulaEigen}
\la_1=\inf\limits_{\substack{\psi\in H^1(\Omega)\\                                                                                                                                                                                                                                                                                                                                                                                                          \Vert \psi\Vert_{{L}^2}=1}} \mathcal{F}(\psi)
:=\inf\limits_{\substack{\psi\in H^1(\Omega)\\                                                                                                                                                                                                                                                                                                                                                                                                          \Vert \psi\Vert_{{L}^2}=1}}\int_\Omega \vert \nabla \psi\vert_A^2-c\psi^2,
\end{equation}
with $\vert \nabla\psi\vert_A^2:=\nabla\psi\cdot A\cdot\nabla\psi.$
\end{proposition}
Note that since the coefficient $c$ is bounded, from the dominated convergence theorem we deduce that the infimum in~\eqref{RayleighFormulaEigen} can be taken equivalently on test functions $\psi\in C^1_0(\overline{\Omega})$.
\begin{proof}
From the dominated convergence theorem and classical density results, it is equivalent to take the infimum on compactly supported smooth test functions in~\eqref{RayleighFormulaEigen}, namely
\begin{equation}
\lambda_1=\inf\limits_{\substack{\psi\in C^1_c(\overline\Omega)\\ \Vert \psi\Vert_{{L}^2}=1}}\mathcal{F}(\psi),
\end{equation}
where $C^1_c(\overline\Omega)$ is the space of continuously differentiable functions with compact support in $\overline\Omega$.
The remaining of the proof is classical and can be adapted from the proof of Proposition 2.2 (iv) in \cite{Berestycki2015b} (which itself relies on \cite{Agmon1983,Berestycki1994}).
\end{proof}

\subsection{The Maximum Principle}
We give, as a complement, some results on the link between the sign of $\lambda_1$ and the validity of the Maximum Principle in unbounded domains. The proofs of the following statements are underlying in the content of the present article, and we leave the details to a forthcoming note~\cite{Nordmann2019}.
\begin{definition}
We say that $(\mathcal{L},\mathcal{B})$ satisfies the Maximum Principle if any subsolution with finite supremum is nonpositive.
\end{definition}
For simplicity, we focus on self-adjoint operators, that is, we assume $B\equiv0$ in~\eqref{DefinitionEllipticOperator}.
We can then express $\lambda_1$ through the Rayleigh-Ritz variational formula~\eqref{RayleighFormulaEigen}.

The first result states that the (strict) sign of $\lambda_1$ is equivalent to the validity of the Maximum Principle.
\begin{proposition}\label{ANG_th:Intro_MPEigen}
Assume $\mathcal{L}$ is self-adjoint.
\begin{enumerate}
\item If $\lambda_1>0$, $(\mathcal{L},\mathcal{B})$ satisfies the Maximum Principle.
\item If $\lambda_1<0$, $(\mathcal{L},\mathcal{B})$ does not satisfy the Maximum Principle.
\end{enumerate}
\end{proposition}
No general answer holds for the degenerate case $\lambda_1=0$. 
Nevertheless, the following proposition states the validity of what could be called a \emph{Critical Maximum Principle} when $\lambda_1\geq 0$ if the domain satisfies a growth condition at infinity.
\begin{proposition}
Suppose that $\mathcal{L}$ is self-adjoint and that the domain $\Omega$ satisfies~\eqref{DomaineGeneral}.
Let $\varphi$ be an eigenfunction associated with $\lambda_1$.
If $\lambda_1\geq0$, then a subsolution with finite supremum is either nonpositive or a multiple of $\varphi$.
\end{proposition}
A first consequence of this result is the simplicity of $\lambda_1$ if it admits a bounded eigenfunction.
\begin{corollary}
Under the same conditions, if $\lambda_1$ admits a bounded eigenfunction, then $\lambda_1$ is simple.
\end{corollary}
The \emph{simplicity} of $\lambda_1$ should be understood as follows: if $\psi\in C^2(\overline\Omega)$ is a solution of~\eqref{EquationGeneralizedEigenfunction}, then it is a scalar multiple of $\varphi$.

We also give the following necessary and sufficient condition for the validity of the Maximum Principle in the critical case $\lambda_1=0$.
\begin{corollary}
Under the same conditions, further assume $\lambda_1=0$, and let $\varphi$ be an associated eigenfunction. Then, $(\mathcal{L},\mathcal{B})$ satisfies the Maximum Principle if and only if $\varphi$ is not bounded.
\end{corollary}

%
%

\section{On the different definitions of stability}\label{Annexe:Stability}
When considering stability from a dynamical point of view, one can come up with the two following definitions. 
\begin{definition}\label{DefinitionDynmicallyStable}
A solution $u$ of \eqref{ANG_Intro_EquationSemilineaire} is said to be dynamically stable if, given any $\eps>0$, there exists $\delta_0>0$ such that for any $v_0(x)$ with $\Vert v_0-u\Vert_{L^\infty}\leq\delta_0$, we have
\begin{equation}
\Vert v(t,\cdot) -u(\cdot)\Vert_{L^\infty}\leq\eps,\quad \forall t>0,
\end{equation}
where $v(t,x)$ is the solution of the evolution problem
\begin{equation}\label{DynamicSystem}
\left\{\begin{aligned}
&\D_tv(t,x)-\Delta v(t,x)=f(v(t,x)) &\forall x\in\Omega,\ \forall t>0,\\
&\D_\nu v(t,x)=0 &\forall x\in\D\Omega,\ \forall t>0,\\
&v(t=0,x)=v_0(x) &\forall x\in\Omega.
\end{aligned}\right.
\end{equation}
\end{definition}

\begin{definition}\label{DefinitionAsymptoticallyStable}
A solution $u$ of \eqref{ANG_Intro_EquationSemilineaire} is said to be asymptotically stable if there exists $\delta_0>0$ such that for any $v_0(x)$ with $\Vert v_0-u\Vert_{L^\infty}\leq\delta_0$, we have
\begin{equation}
\Vert v(t,\cdot) -u(\cdot)\Vert_{L^\infty}\to 0,\quad\text{when }t\to+\infty,
\end{equation}
where $v(t,x)$ is the solution of \eqref{DynamicSystem}.
\end{definition}
The following proposition clarifies the hierarchy of the different definitions of stability.

\begin{proposition} \label{PropositionStability}
Let $u$ be a solution of \eqref{ANG_Intro_EquationSemilineaire} and $\lambda_1$ from \eqref{DefLambda}. The following implications hold
\begin{equation}
\text{ $u$ asymptotically stable}\Rightarrow \text{ $u$ dynamically stable}\Rightarrow \lambda_1\geq 0.
\end{equation}
\end{proposition}

\begin{proof}
The first implication is trivial. 
Let us show the second implication by contradiction: assume $\lambda_1<0$ and that $u$ is dynamically stable. 
For $R>0$, define the truncated domain $\Omega_R:=\Omega\cap\{\vert x\vert <R\}$ and consider the following mixed-boundary eigenvalue problem: find $\lambda_{1,R}\in\R$ and $\varphi_R\in C^2(\Omega_R)$ satisfying
\begin{equation}\label{MixedBoundaryEigenproblem}
\left\{\begin{aligned}
&-\Delta\varphi_R-f'(u)\varphi_R=\lambda_{1,R} \varphi_R&&\text{in }\Omega_R,\\
&\D_\nu\varphi_R=0&&\text{on }\D\Omega\cap\{\vert x\vert <R\},\\
&\varphi_R=0&&\text{on }\Omega\cap \{\vert x\vert =R\}.
\end{aligned}\right.
\end{equation}
From a recent result of Rossi~\cite[Theorem~2.1]{Rossi2020}, we know that, for almost every $R>0$, the eigenproblem~\eqref{MixedBoundaryEigenproblem} admits a unique eigenpair $(\lambda_{1,R},\varphi_R)$ such that $\varphi_R>0$ on $\overline{\Omega}\cap\{\vert x\vert <R\}$. Moreover, $R\mapsto \lambda_{1,R}$ is strictly decreasing and $\lim_{R\to+\infty}\lambda_{1,R}=\lambda_1$.
Let us fix $R>0$ large enough such that $\lambda_{1,R}<0$. We also choose the normalization $\Vert \varphi_R\Vert_{L^\infty}=1$.

Let us consider a parameter $\eps>0$ small enough such that
\begin{equation}\label{DefinitionEtaStable}
\eta_\eps:=\sup\limits_{\substack{\tilde u\in[\inf_{\Omega_R} u,\sup_{\Omega_R} u]\\ \vert h\vert\leq \eps}} \left\vert f'(\tilde u)-\frac{f(\tilde u+h)-f(\tilde u)}{h}\right\vert< -\lambda_{1,R},
\end{equation}
and $\delta_0$ given by \autoref{DefinitionDynmicallyStable}. Consider $v$ the solution of the parabolic equation~\eqref{DynamicSystem} with initial datum $v_0:=u+\delta_0\varphi_R$, and set $h(t,x)=v(t,x)-u(x)$. On the one hand, since $\Vert v_0-u\Vert_{L^\infty}\leq\delta_0$, the stability assumption implies $\Vert h(t,\cdot)\Vert_{L^\infty}\leq\eps$ for all time $t\geq0$. On the other hand, $h$ satisfies
\begin{equation*}
\left\{\begin{aligned}
&\D_t h(t,x)- \Delta h(t,x)\geq \left(f'(u(x))-\eta_\eps\right)h(t,x) &&\text{in }\Omega_R,\\
&\D_\nu h=0&&\text{on }\D\Omega\cap\{\vert x\vert <R\},\\
&h\geq 0&&\text{on }\Omega\cap\{\vert x\vert =R\}
\end{aligned}\right.
\end{equation*}
From the parabolic comparison principle, we infer $h(t,x)\geq \tilde h(t,x):= e^{-(\lambda_1+\eta_\eps)t}\delta_0\varphi_R(x)$ for all $t\geq0$ and $x\in\Omega_R$. Using that $\lambda_1+\eta_\delta<0$, we deduce that $\Vert h(t,\cdot)\Vert_{L^\infty}$ diverges to $+\infty$ when $t$ becomes large: contradiction.
\end{proof}
\begin{remark}
Note that, in the proof, the perturbation $\delta_0 \varphi_R$ has a compact support in $\overline\Omega$ and an arbitrarily small $L^\infty$ norm. Thus, if $\lambda_1<0$ then~\eqref{DynamicSystem} drives $u+h$ away from $u$ for \emph{any} $h$ which is positive or negative on $\Omega_R$ if $R$ is large enough. 
\end{remark}

One can ask whether the following implication holds:
\begin{equation}\label{QuestionStrongStability}
\lambda_1 >0 \Rightarrow u\text{ asymptotically stable}.
\end{equation}
This implication is classical when the domain is bounded (see Proposition 1.4.1 in \cite{Dupaigne2011}), but it is not clear whether it extends to unbounded domains.
\begin{question}Does \eqref{QuestionStrongStability} hold in unbounded domains ?
\end{question}
We think that, in general, the answer is negative. Nevertheless, as a consequence of the results of the present paper, we give a positive answer for unbounded \emph{convex} domains.

\begin{proposition}\label{CorollaryFull}
Let $\Omega\subset\R^n$ be a smooth convex domain (possibly unbounded) and $u$ be a solution of \eqref{ANG_Intro_EquationSemilineaire}. Then
\begin{equation}
\lambda_1>0\Rightarrow\text{ $u$ asymptotically stable}\Rightarrow \text{ $u$ dynamically stable}\Rightarrow \lambda_1\geq 0.
\end{equation}
\end{proposition}

\begin{proof}
From \autoref{PropositionStability}, we only have to show the first implication. Assume $\lambda_1>0$.
We deduce from \autoref{thmStableNonDegenerate} that $u$ is constant. Thus $\lambda_1=-f'(u)$ and $\varphi$ is constant. We choose $\eps$ small enough such that $\eta_{\eps}\in(0,\frac{\lambda_1}{2})$ with $\eta_\eps$ defined in \eqref{DefinitionEtaStable}. Let $v_0$ be as in \autoref{DefinitionAsymptoticallyStable} (we use the same notations).
We set $T:=\sup\{t>0:\Vert h(t,\cdot)\Vert_{L^\infty}\leq \eps\}$. By continuity and the choice of $v_0$, we know that $T>0$. We have
\begin{equation}
\left\{\begin{aligned}
&\D_t h(t,x)-\Delta h(t,x)\leq \left(f'(u)+\eta_\eps\right)h(t,x) &\forall t\in(0,T),\ x\in\Omega,\\
&\D_\nu h(t,x)=0 &\forall t\in(0,T),\ x\in\D\Omega.
\end{aligned}\right.
\end{equation}
From $f'(u)+\eta_\eps\leq -\frac{\lambda_1}{2}$ and the parabolic comparison principle, we obtain $\Vert h(t,\cdot)\Vert_{L^\infty}\leq \eps e^{-\frac{\lambda_1}{2}t}$, for all $t\in(0,T)$. We deduce $T=+\infty$ and $\Vert h(t,\cdot)\Vert_{L^\infty}\to0$ when $t\to+\infty$, thus $u$ is asymptotically stable.
\end{proof}

\section{Isolation of stable solutions}\label{sec:Isolation}
We give a brief discussion on the isolation of stable solutions of~\eqref{ANG_Intro_EquationSemilineaire} in the set of all solutions.
Note that this question is crucial in the proof of \autoref{th:AsymptoticNonDegenerate} and \autoref{th:AsymptoticDegenerate}, since the key point is to show that $\Sigma_u$ is a discrete set.
When considering the $L^\infty$ topology, we have the following.
\begin{lemma}\label{LemmaIsolated}
Let $\Omega\subset\R^n$ be a smooth domain (possibly unbounded), and denote $S$ the set of solutions of \eqref{ANG_Intro_EquationSemilineaire} in $\Omega$. Let $u\in S$ be either stable non-degenerate, or unstable non-degenerate (i.e. $\lambda_1<0$). Then, $u$ is isolated in $S$ for the $L^\infty(\Omega)$ topology.
\end{lemma}
This result is essentially classical, at least for bounded domains. We give a proof at the end of the section.

Note that, in general, this result fails for $L^\infty_{loc}$ topology. For example, consider the Allen-Cahn equation in $\R$,
$
        -u''=u(1-u^2).
$
This equation admits an explicit solution $u : x\mapsto \tanh\frac{x}{\sqrt{2}}$
which is stable (degenerate). On the one hand, the family of the translated solutions $u_a(\cdot)=u(\cdot-a)$ converges to $0$ when $a\to+\infty$ in the $L^\infty_{loc}$ topology. On the other hand, $0$ is a stable non-degenerate solution (because $f'(0)<0$). 

One could argue that the above counterexample relies on the fact that the nonlinearity is balanced, that is, $\int_0^1 f=0$ with $f(u):=u(1-u)(u-\nicefrac{1}{2})$.
However, one can build similar counterexamples for unbalanced nonlinearities by considering a ground state (which has been proved to exist in most cases, see, e.g., \cite{Berestyckil}).

\paragraph*{}
In an attempt to extend the results of section~\ref{sec:AsymptoticFormulation}, we address the following question.
\begin{question}
Is $\Sigma_u$ from~\eqref{DefinitionOmegaLimitSet} always a singleton when $u$ is stable non-degenerate?
\end{question}
We think the answer is negative; yet, we are not able to provide a counterexample.

\begin{proof}[\autoref{LemmaIsolated}]
Assume that there exists $u_k\in S$, $u_{k+1}\not\equiv u_k$, a sequence which converges to $u$, and let us show that $\lambda_1=0$. We set $v_k:=u_{k+1}-u_k$. For all $k$, $u_k$ is a solution of~\eqref{ANG_Intro_EquationSemilineaire} in $\Omega$, thus
\begin{equation}\label{EquationV_k}
\left\{\begin{aligned}
&-\Delta v_k-c_k(x)v_k=0&&\text{in }\Omega,\\
&\D_\nu v_k=0&&\text{on }\D\Omega,
\end{aligned}\right.
\end{equation}
where
\begin{equation}
c_k(x):=\frac{f(u_{k+1}(x))-f(u_k(x))}{u_{k+1}(x)-u_k(x)}.
\end{equation}
Since $f$ is $C^{1,\alpha}$ and $u_{k+1}-u_k$ is bounded, $c_k(x)$ converges uniformly to $f'(u(x))$ when $k\to+\infty$.

Formally, we have $$\mathcal{F}\left(\frac{v_k}{\Vert v_k \Vert_{\substack{{L}^2}}}\right)\leq \Vert f'(u)-c_k\Vert_\infty\underset{k\to0}{\longrightarrow}0,$$ (with $\mathcal F$ from \eqref{DefLambda}) which contradicts the fact that $u$ is stable non-degenerate. However, the former calculation is not licit when $\Omega$ is unbounded. To make it rigorous, we use the cut-off function $\chi_R$ defined in~\eqref{DefCutOff}.

Multiplying \eqref{EquationV_k} by $v_k\chi_R^2$, integrating on $\Omega$, using the divergence theorem and the boundary condition in \eqref{EquationV_k} we find
\begin{align*}
\mathcal{F}\left(\frac{\chi_Rv_k}{\Vert \chi_Rv_k\Vert_{\substack{{L}^2}}}\right)
&=\frac{\int_{\substack{\Omega}}\chi_R^2v_k^2 (c_k-f'(u))}{\int_{\substack{\Omega}}\chi_R^2v_k^2}+\frac{\int_{\substack{\Omega}} \vert\nabla\chi_R\vert^2v_k^2}{\int_{\substack{\Omega}}\chi_R^2v_k^2}\\
&\leq\Vert c_k-f'(u)\Vert_{L^\infty\left(\Omega_R\right)} +4\ \alpha_{k,R},
\end{align*}
where
\begin{align*}
\mathcal{C}_k(R):=\int_{\Omega_R} v_k^2,\quad \alpha_{k,R}:= \frac{\mathcal{C}_k(2R)}{R^2\mathcal{C}_k(R)}.
\end{align*}
On the one hand, in the proof of \autoref{thmStableNonDegenerate}, we show that, for fixed $k\geq0$, $\liminf\limits_{R\to+\infty} \alpha_{k,R}\leq 0.$
On the other hand, since $\Vert c_k-f'(u)\Vert_{L^\infty\left(\Omega_R\right)}$ goes to $0$ when $k\to+\infty$, uniformly in $R$, we deduce that $\mathcal{F}$ can be made arbitrarily small. It implies $\lambda_1\leq0$.
The reverse inequality $\lambda_1\geq0$ can be proved with the same method, which achieves the proof.
\end{proof}

\bibliographystyle{abbrv}
\bibliography{library}
\end{document}